\documentclass[11pt,reqno]{amsproc}
\usepackage[margin=1in]{geometry}
\usepackage{amsmath, amsthm, amssymb}
\usepackage{times, esint,stackrel}
\usepackage{enumitem}
\usepackage{color}  
\usepackage[colorlinks=true]{hyperref}
\usepackage[color=yellow]{todonotes}
\usepackage{mathtools}
\usepackage{etoolbox}
\usepackage{mathrsfs}
\usepackage{framed}

\makeatletter
\patchcmd\@thm
  {\let\thm@indent\indent}{\let\thm@indent\noindent}%
  {}{}
\makeatother

\usepackage{etoolbox}
\expandafter\patchcmd\csname\string\proof\endcsname
  {\normalparindent}{0pt }{}{}

\newcommand{\be}{\begin{equation}}
\newcommand{\ee}{\end{equation}}
\newcommand{\bea}{\begin{eqnarray}}
\newcommand{\eea}{\end{eqnarray}}

\newtheorem{thm}{Theorem}
\newtheorem{prop}{Proposition}
\newtheorem{lemma}{Lemma}

\theoremstyle{definition}
\newtheorem{rem}{Remark}

\newcommand{\rmd}{{\rm d}}

\newcommand{\ol}[1]{\mkern 1.5mu\overline{\mkern-1.5mu#1\mkern-1.5mu}\mkern 1.5mu}

\def\curl{\text{curl}}

\newcommand{\bq}{\begin{equation}}
\newcommand{\eq}{\end{equation}}
\newcommand{\bqa}{\begin{eqnarray*}}
\newcommand{\eqa}{\end{eqnarray*}}

\newcommand{\tr}{{\rm tr}}

\title[Circulation and Energy Theorem Preserving Stochastic Fluids]{\vspace{-22.5mm}Circulation and Energy Theorem Preserving Stochastic Fluids}
\author{Theodore D. Drivas}
\address{Department of Mathematics, Princeton University, Princeton, NJ 08544}
\email{tdrivas@math.princeton.edu}
\author{Darryl D. Holm}
\address{Department of Mathematics, Imperial College, London SW7 2AZ, UK.}
\email{d.holm@imperial.ac.uk }

\date{today}
\begin{document}

\begin{abstract}
Smooth solutions of the incompressible Euler equations are characterized by the property that circulation around material loops is conserved. This is the Kelvin theorem  \cite{Kelvin69}.   Likewise, smooth solutions of Navier-Stokes are characterized by a generalized Kelvin's theorem, introduced by Constantin--Iyer (2008) \cite{CI08}.   In this note, we introduce a class of stochastic fluid equations, whose smooth solutions are characterized by natural extensions of the Kelvin theorems of their deterministic counterparts, which hold along certain noisy flows.  These equations are called the \emph{stochastic Euler--Poincar\'{e}} and \emph{stochastic Navier-Stokes--Poincar\'{e}} equations respectively. The stochastic Euler--Poincar\'{e} equations were previously derived from a stochastic variational principle by Holm (2015) \cite{DH15}, which we briefly review.
Solutions of these equations do not obey pathwise energy conservation/dissipation in general. In contrast, we also discuss a class of stochastic fluid models, solutions of which possess energy theorems, but do not, in general, preserve circulation theorems.
\end{abstract}

\vspace*{23mm}
\maketitle

%\tableofcontents

\vspace{-6mm}

 \section{Introduction}

In 1869, Lord Kelvin (Sir William Thomson) \cite{Kelvin69} discovered a beautiful property of smooth solutions of the incompressible Euler equations. Namely, the circulation of velocity around any closed loop advected by an ideal fluid is conserved.  More precisely, let the spatial domain of flow be $\Omega=\mathbb{T}^d$ or $\mathbb{R}^d$, and  suppose the fluid velocity $u_t:= u(x,t):\Omega\times [0,T]\to \mathbb{R}^d$ solves the incompressible Euler equations,
\begin{align}\label{ee1}
\begin{split}
\partial_t u_t + (u_t\cdot  \nabla)  u_t &= -\nabla p_t,\\
\nabla \cdot u_t&= 0, \\
u_t|_{t=0} &= u_0,
\end{split}
\end{align}
with scalar pressure function $p_t$, determined by solving the Poisson equation
\be\label{presspoiss}
-\Delta p_t = (\nabla\otimes \nabla): (u_t\otimes    u_t ),
\ee
which enforces incompressibility at each time, $t$.
The \emph{Kelvin theorem} states that any smooth Euler solution $u_t$ has the property that for all loops $\Gamma \subset \Omega$, the circulation integral satisfies, 
\be\label{KelvinTheorem}
\oint_{X_t(\Gamma)} u_t \cdot \rmd \ell = \oint_\Gamma u_0\cdot \rmd \ell\,,
\ee
where $X_t$ is Lagrangian flow satisfying $\dot{X}_t= u_t(X_t)$,
$X_{0} ={\rm id}$. 
 The  Kelvin Theorem offers an elegant interpretation of the Lagrangian laws of vortex motion written down by Helmholtz in 1858 \cite{Helmholtz58}.  
   
Remarkably, the converse implication also holds. That is, any sufficiently regular incompressible velocity field possessing the property \eqref{KelvinTheorem} for all times $t\in [0,T]$ and for any closed, rectifiable loop $\Gamma\subset\Omega$ must, in fact, be a smooth Euler solution. This follows readily from \eqref{KelvinTheorem}, since its time derivative implies that
    \be\label{transthm}
   \oint_{X_t(\Gamma)} \big(\partial_t u_t + (u_t\cdot \nabla )u_t   \,+\, (\nabla u_t)^T\cdot  u_t  \big) \cdot \rmd \ell = 0\,
   \ee   
for all loops $\Gamma$ and all times $t\in [0,T]$.  In particular, Eq. \eqref{transthm} holds for the rectifiable loop $\Gamma = X_{t'}^{-1}(\Gamma')$ for any fixed $t'\in [0,T]$  (since $X_t$ is a diffeomorphism).  For such a loop, evaluating Eq. \eqref{transthm} at time $t=t'$ shows that the line integral vanishes when $X_{t}(\Gamma)$ above is replaced by an \emph{arbitrary} loop $\Gamma'$.  
From Stokes theorem, Eq. \eqref{transthm} holds for all loops $\Gamma'\subset \Omega$, if and only if there exists a scalar function $\pi_t = \pi(x,t)$ such that the integrand is equal to the gradient of this potential $\nabla \pi_t$ for all $(x,t)\in \Omega\times [0,T]$. Then, using the identity $(\nabla u_t)^T\cdot u_t = \nabla |u_t|^2/2$, one finds that Eq. \eqref{ee1} holds with $p_t$ replaced by $q_t:=\frac{1}{2}|u_t|^2 - \pi_t$.  Finally, to enforce incompressibility of $u_t$, the scalar function $q_t$  solves Eq. \eqref{presspoiss} thus fixing it as the pressure $q_t= p_t$ (up to a constant).
Therefore, one may say that smooth solutions of the Euler equations in the domain $\Omega$ are \emph{characterized} by the Kelvin theorem.  It is worth noting that this equivalence was already realized by Lord Kelvin in his original 1869 paper \cite{Kelvin69}.
    
    Kelvin's theorem has long been recognized as centrally important to the understanding of deterministic, smooth, ideal fluid dynamics. Its geometric meaning is discussed in Appendix \ref{app-GeometBckgrnd}. One might then ask the following question: How would Kelvin's theorem be changed, if the fluid flow were stochastic?\smallskip

\subsubsection*{Stochastic Euler--Poincar\'{e} equations.}
In 2015, Holm \cite{DH15} introduced a family of stochastic partial differential equations (SPDEs) for fluid dynamics, whose smooth solutions possess a certain pathwise Kelvin theorem.  These equations arise from a stochastic variational principle, which we review in Appendix \ref{appendVar}.

Specifically,  let $(\Xi,\mathcal{F}, {\bold P})$ be a probability space with a filtration $\mathbb{F}$
of right continuous $\sigma$-algebras $(\mathcal{F}_t)_{t\geq 0}$. All of the $\sigma$-algebras are assumed to be $\mathbf{P}$--completed. Let  $\{W_t^{(k)}\}_{k\in \mathbb{N}}$ be a collection of  $\mathcal{F}_t$-adapted  independent 1-dimensional Brownian motions  in $\mathbb{R}$.  
The Stratonovich form of the circulation-theorem preserving stochastic Euler--Poincar\'{e} equations introduced in \cite{DH15} for the 1-form $u_t$ read
 \be\label{e1}
 \rmd u_t  + \mathbb{P}(\pounds_{u_t}^T u_t ) \rmd t + \sum_{k}   \mathbb{P}(\pounds_{\xi^{(k)}}^T u_t ) \circ \rmd W_t^{(k)}  = 0
 \,,\quad\hbox{with}\quad
 u_t|_{t=0} = u_0 \,.
 \ee
Here, $ \mathbb {P}$ is the dual for 1-forms of the standard Leray projection operator for vector fields, as discussed in Appendix \ref{app-GeometBckgrnd}. The symbol $\circ$ denotes the Stratonovich sense of the stochastic product, the collection $\{\xi^{(k)}\}_{k\in \mathbb{N}}$ contains  fixed, deterministic divergence-free vector fields $\xi^{(k)}: \Omega\mapsto \mathbb{R}^d$, and we define the operator $\pounds_v^T$ as
\begin{align}\label{e2}
\begin{split}
\pounds^T_{v} u_t  &:=  v\cdot \nabla u_t + (\nabla v)^T \cdot u_t 
\\&= - \,v \times {\rm curl} u_t + \nabla (v\cdot u_t)
\quad \hbox{in 3D.}
\end{split}
\end{align}
In index notation, $(\pounds^T_{v} u)_i= v^j\partial_j u_i + (\partial_i v^j)u_j$. The geometric justification for choosing to write the nonlinearity in the $\pounds^T$ form is explained in Appendix \ref{app-GeometBckgrnd}.  Smooth solutions of the stochastic Euler--Poincar\'{e} equations in Eqns. \eqref{e1} possess a stochastic Kelvin theorem, which we describe in Theorem \ref{KelvinTheorem} below. 
\vspace{-1mm}

 \section{Main Results}
In this paper, we prove that the stochastic Euler--Poincar\'{e} equations are, in fact, characterized by the pathwise Kelvin theorem \eqref{KelvinTheorem}.
For this purpose,  we consider a class of abstract stochastic  It\^o SPDEs 
 \begin{align}\label{uEqn1}
 \rmd u_t +\mathbb{P} f_t \rmd t +\sum_{k} \mathbb{P}   \sigma_t^{(k)}    \rmd W_t^{(k)} &=0 \,,\quad\hbox{with}\quad
 u_t|_{t=0} = u_0 \,.
\end{align}
 The system  \eqref{uEqn1} maintains incompressibility $\nabla \cdot u_t =0$ while its solutions exist.   
Equation \eqref{uEqn1} for the 1-form $u_t$ is to be understood in the weak sense: for any solenoidal test vector field $\varphi\in C_0^\infty(\Omega)$, the following equality holds
\be\label{weakform}
\langle u_t, \varphi\rangle_{L^2}= \langle u_0, \varphi\rangle_{L^2} - \int_0^t\langle f_t, \varphi\rangle_{L^2}\rmd t- \sum_{k} \int_0^t     \langle  \sigma_t^{(k)}, \varphi\rangle_{L^2}\rmd W_t^{(k)} .
\ee
We prove the following result.

 \begin{thm}[Characterization of Stochastic Euler--Poincar\'{e} Fluids]\label{ThmKE}
Let $X_t$ be the flow defined by the SDE
  \be\label{flow}
 \rmd X_t(x) = u_t(X_t(x)) \rmd t +\sum_{k} \xi^{(k)}(X_t(x)) \circ  \rmd W_t^{(k)}  , \qquad X_0(x)=x\,,
 \ee
 for fixed smooth solenoidal vector fields $u_t: [0,T] \times \Omega \mapsto\mathbb{R}^d$ and  $\{\xi^{(k)}\}_{k\in \mathbb{N}}:\Omega\mapsto \mathbb{R}^d$.
 
Then, $u_t$ is a smooth solution of \eqref{uEqn1} on $[0,T] \times \Omega$ with 
 \begin{align}\label{feq}
 f_t &=  \pounds_{u_t}^T   u_t -\sum_k\frac{1}{2} \pounds_{\xi^{(k)}}^T( \pounds_{\xi^{(k)}}^T u_t)
 \,,\\
  \sigma_t^{(k)} &= \pounds_{\xi^{(k)}}^T u_t
  \,,\label{sigeq}
   \end{align}
   if and only if, for every rectifiable loop $\Gamma\subset \Omega$, $u_t$ has the property that for all $t\in [0,T]$, 
   \be\label{KelvinThm}
   \oint_{X_t(\Gamma)} u_t \cdot \rmd \ell =    \oint_{\Gamma} u_0 \cdot \rmd \ell, \quad    \mathbf{P} \ a.s.
   \ee
 \end{thm}

\begin{rem}[It\^o form of the  Stochastic Euler--Poincar\'{e} Equations]
The It\^o form of equation \eqref{e1} reads
  \begin{align}\label{e1Ito}
 \rmd u_t + \mathbb{P} \Big(   \pounds_{u_t}^T u_t - \frac{1}{2}  \sum_k \pounds_{\xi^{(k)}}^T( \pounds_{\xi^{(k)}}^T u_t)\Big)  \rmd t + \sum_{k} \mathbb{P}   \pounds_{\xi^{(k)}}^T u_t   \rmd W_t^{(k)} &=0.
\end{align}
This follows from the Stratonovich-to-It\^o conversion
\begin{align*}
\int_0^t \mathbb{P}   \pounds_{\xi^{(k)}}^T u_s \circ  \rmd W_s^{(k)}  &= \int_0^t \mathbb{P}   \pounds_{\xi^{(k)}}^T u_s   \rmd W_s^{(k)}  + \frac{1}{2} \left[  \mathbb{P}   \pounds_{\xi^{(k)}}^T u , W^{(k)}\right]_t\\
&= \int_0^t \mathbb{P}   \pounds_{\xi^{(k)}}^T u_s   \rmd W_s^{(k)}  - \frac{1}{2} \int_0^t \mathbb{P}   \pounds_{\xi^{(k)}}^T(\mathbb{P}   \pounds_{\xi^{(k)}}^T u_s) \rmd s,
\end{align*}
where $\left[\cdot, \cdot\right]_t$ denotes the quadratic variation and where we have used Eqn. \eqref{e1} to compute this cross-variation.    To obtain \eqref{e1Ito}, we note that for any 1-form $v$, we have $\mathbb{P}\pounds_{\xi}^T\mathbb{P} v= \mathbb{P}\pounds_{\xi}^T v$.  To see this, write   $\mathbb{P} v= v+\nabla q$ for some scalar $q$ and note that $\pounds_{\xi}^T\nabla q$ is a gradient,
\be\label{vanishingOnGrad}
\pounds_{\xi}^T\nabla q = (\xi \cdot \nabla)\nabla q  + \nabla \xi \cdot \nabla q= \nabla (\xi\cdot \nabla q).
\ee
Hence, $\mathbb{P} \pounds_{\xi}^T\nabla q=0$ and equation \eqref{e1Ito} follows.  In view of Theorem \ref{ThmKE}, we recover the stochastic Euler--Poincar\'{e} equations \eqref{e1Ito} as the unique equations for which smooth solutions obey pathwise circulation conservation along the stochastic flow \eqref{flow}. 
\end{rem}

  \begin{rem} [{Regularity of Flow}] \label{remReg}
Provided that $\sum_{k} \|\xi^{(k)}\|^2_{C^{n+3,\alpha'}(\Omega)}<\infty$ for some  $n\in \mathbb{N}$   and  $\alpha'\in (0,1)$, and $u\in C(0,T; C^{n+1,\alpha}(\Omega))$  for any $\alpha\in (0,\alpha')$, then equation \eqref{flow} generates a flow of $C^{n+1,\alpha}$--diffeomorphisms of $\Omega$ \cite{K97,LJW84}.  Moreover, the inverse map $A_t:= X_t^{-1}$  exists and belongs to the same space $C(0,T; C^{n+1,\alpha}(\Omega))$ and the gradient belongs to $\nabla X_t\in  C(0,T; C^{n,\alpha}(\Omega))$.  This is sufficient regularity to justify the computations of the present paper, in particular the use of the It\^{o}--Wentzell formula \cite{K97,K81}.
    \end{rem}
    
 \begin{rem} [Local Existence and Regularity for Euler--Poincar\'{e} Fluids]\label{localExRem}
Well--posedness for equations \eqref{uEqn1} with \eqref{feq} and \eqref{sigeq} has recently been established in \cite{CFH17}. In \S 3.3 of \cite{CFH17}, it is shown there exists (for data in the appropriate Sobolev space) a maximal stopping time $\tau_{max}: \Xi\mapsto [0,\infty)$ and a unique solution $u\in C(0,\tau;W^{3,2}(\mathbb{T}^3;\mathbb{R}^3))$ for all $\tau\leq \tau_{max}$.  Subsequently,  \cite{F18}  established local existence of \eqref{e1}--\eqref{e2} in H\"{o}lder spaces $C(0,T; C^{n+1,\alpha}(\Omega))$  for some  $k\in \mathbb{N}$ and some $\alpha\in (0,1)$, by using the Weber formula \eqref{weber} and following the Eulerian--Lagrangian scheme of Constantin \cite{C01}. Thus, in view of Remark \ref{remReg}, regularity in the appropriate H\"{o}lder spaces can be taken as the precise meaning of ``smooth" in Theorem \ref{ThmKE} as well as in Proposition \ref{propCirc} and Theorem \ref{ThmKENS} appearing below. 
 \end{rem}

The key to the proof of Theorem \ref{ThmKE} is a general formula for the transport of circulations along the stochastic flow \eqref{flow}, where the velocity $u_t$ is a stochastic process driven by the same Brownian noise $\{W_t^{(k)}\}_{k\in \mathbb{N}}$.

   \begin{prop}[Stochastic Circulation Transport]\label{propCirc}
 Fix smooth vector fields $\xi^{(k)}:\Omega\mapsto \mathbb{R}^{d}$. Let $u_t: [0,T]\times \Omega\mapsto \mathbb{R}^d$ be a smooth solution of equation \eqref{uEqn1} and  $X_t:=X_t(x)$ be the stochastic flow defined by the SDE
  \be\label{flow2}
 \rmd X_t(x) = u_t(X_t(x)) \rmd t +\sum_{k} \xi^{(k)}(X_t(x))\circ  \rmd W_t^{(k)}  , \qquad X_0(x)=x.
 \ee
Then, for any rectifiable loop $\Gamma\subset \Omega$, the following holds for $t\in[0,T]$
\begin{align}\nonumber
\rmd \oint_{X_t(\Gamma)} u_t\cdot \rmd \ell  &=  \oint_{X_t(\Gamma)}  \left(\pounds_{u_t}^T   u_t -f_t \right)  \rmd t  \cdot \rmd \ell \\\nonumber
          &\quad +  \sum_k \oint_{X_t(\Gamma)} \left(\frac{1}{2} \pounds_{\xi^{(k)}}^T( \pounds_{\xi^{(k)}}^T u_t)-  \pounds_{\xi^{(k)}}^T\sigma_t^{(k)} \right)  \rmd t  \cdot \rmd \ell \\
  &\quad + \sum_k   \oint_{X_t(\Gamma)} \left( \pounds_{\xi^{(k)}}^T u_t - \sigma_t^{(k)}  \right) \rmd W_t^{(k)}\cdot \rmd \ell.
   \label{circulations}
\end{align}
 \end{prop}
 
  \begin{rem} [It\^{o}--Wentzell formula]
The noise appearing in the flow \eqref{flow2} is the same noise that drives the stochastic evolution of $u_t$.  Consequently, these objects are correlated and to compute the rate of change of circulation we employ the It\^{o}--Wentzell formula. This formula results in the presence of the term $\pounds_{\xi^{(k)}}^T\sigma_t^{(k)} $ in the second line of equation \eqref{circulations}.
  \end{rem}

 \begin{rem} [Pathwise Kelvin Theorem along  It\^o flow]
One could consider loops transported by the stochastic flow with It\^o noise instead of Stratonovich,
   \be
 \rmd Y_t(x) = u_t(Y_t(x)) \rmd t +\sum_k \xi^{(k)}(Y_t(x))   \rmd W_t^{(k)}, \qquad Y_0(x)=x.
 \ee 
An argument similar to that made to prove  Prop. \ref{propCirc} shows that for any rectifiable loop $\Gamma$, the following holds
\begin{align}\nonumber
\rmd \oint_{Y_t(\Gamma)} u_t\cdot \rmd \ell  &=  \oint_{Y_t(\Gamma)}  \left(\pounds_{u_t}^T   u_t  -f_t \right)  \rmd t  \cdot \rmd \ell \\\nonumber
                    &\quad +  \sum_k \oint_{Y_t(\Gamma)} \left(\  \frac{1}{2}\xi^{(k)} \otimes \xi^{(k)} : \nabla \otimes \nabla u_t +  \nabla \xi^{(k)} \cdot (\xi^{(k)} \cdot \nabla) u_t-  \pounds_{\xi^{(k)}}^T\sigma_t^{(k)} \right)  \rmd t  \cdot \rmd \ell \\
  &\qquad + \sum_k   \oint_{Y_t(\Gamma)} \left( \pounds_{\xi^{(k)}}^T u_t - \sigma_t^{(k)}  \right) \rmd W_t^{(k)}\cdot \rmd \ell. \label{circulationsIto}
  \end{align}
Formula \eqref{lieComps} in the proof below exhibits the terms that are omitted in  \eqref{circulationsIto} relative to the full double-Lie diffusion appearing in equation \eqref{circulations}.  Consequently, this alternative pathwise Kelvin theorem provides a characterization of smooth solutions $u_t$  to \eqref{uEqn1} with the same noise $  \sigma_t^{(k)} = \pounds_{\xi^{(k)}}^T u_t$ but with the drift
 \begin{align}\label{lieF}
 &f_t = \pounds_{u_t}^T   u_t \\
 &\  -\sum_k\left( \frac{1}{2}\xi^{(k)} \otimes \xi^{(k)} : \nabla \otimes \nabla u_t +  \nabla \xi^{(k)} \cdot (\xi^{(k)} \cdot \nabla) u_t + (\xi^{(k)}\cdot \nabla ) \xi^{(k)}\cdot \nabla u_t   +\nabla ( (\xi^{(k)}\cdot \nabla)\xi^{(k)}) \cdot u_t\right).\nonumber
   \end{align}
  Comparing \eqref{lieF} with \eqref{feq}, we see that along the It\^o flow $Y_t$ the double-Lie diffusion structure in the SPDE is lost. See \S 1 of \cite{DH15} for a discussion of a similar issue; quantification of the failure of circulation conservation along loops which are advected by  \emph{It\^{o}} flow given the velocity satisfies an SPDE, smooth solutions of which conserve circulation along loops evolving according to the flow with \emph{Stratonovich} noise. 
 \end{rem}
 
 \begin{rem}[Pathwise Weber Formula] \label{WeberRemark}
 A simple consequence of the calculations used in proofs of Theorem  \ref{ThmKE} and Proposition \ref{propCirc} is that smooth solutions of the stochastic Euler--Poincar\'{e} equations  \eqref{e1}--\eqref{e2}  satisfy a pathwise {\it Weber formula}:
\be\label{weber}
 u^\sharp_t(x)= \mathbb{P} \left[(\nabla A_t(x))^T u_0(A_t(x))\right]^\sharp, \quad   \mathbf{P} \ a.s.
\ee
where $A_t= X_t^{-1}$ is the ``back-to-labels" map and $X_t$ is the stochastic flow defined by \eqref{flow} and, when taking projection $\mathbb{P}$, there is an implied transformation from 1-forms to vector fields by the operation $\sharp$. See Appendix \ref{app-GeometBckgrnd} for more explanation of the notation $\sharp$.  
For a proof of the representation \eqref{weber}, see \cite{F18}.  This result can be expressed also at the level of vorticity $\omega_t = {\rm curl} (u_t)$, where one has 
an exact Cauchy formula of the form
\be\label{CauchyFomula}
\omega_t(x) = ((\nabla X_t)\omega_0)\circ A_t(x), \quad   \mathbf{P} \ a.s.
\ee
Cauchy's vorticity representation in \eqref{CauchyFomula} elucidates what is already apparent directly from \eqref{KelvinThm}. Namely, the circulation theorem may be expressed, using Stokes theorem, in terms of the flux of vorticity through advected areas.  Specifically, letting $S$ be any smooth bounding surface of the closed loop $\Gamma$ with $\Gamma=\partial S$, we have 
   \be\label{KelvinThmVort}
   \oint_{X_t(S)} \omega_t \cdot \rmd S=    \oint_{S} \omega_0 \cdot \rmd S, \quad    \mathbf{P} \ a.s.
   \ee
 \end{rem}

\begin{rem} [Pathwise Energy Preserving Stochastic Fluids]  \label{remEnergy1}
In general for non-constant $\{\xi^{(k)}\}_{k\in \mathbb{N}}$, the Equations \eqref{e1}--\eqref{e2} do not conserve energy, neither pathwise, nor in expectation.  See Remark \ref{remEnergy2} for more details.
Here, we briefly consider a class of stochastic fluid equations that, by design, conserves energy pathwise.  These can be expressed with Stratonovich noise in terms of the operator $B(w,v)= \mathbb{P}(w \cdot \nabla v)$ as 
\be\label{stratencons}
\rmd u_t + B(u_t\rmd t + \sum_k  \xi^{(k)}  \circ \rmd W_t^{(k)} , u_t).
\ee
 In It\^{o} form, Eqn. \eqref{stratencons} reads as  Eqn. \eqref{uEqn1} with
 \begin{align}\label{Econseqnsf} 
 f_t&:= u_t \cdot \nabla u_t -    \sum_k\xi^{(k)}  \cdot \nabla\mathbb{P} ( \xi^{(k)}  \cdot \nabla   u_t), \\
   \sigma_t^{(k)}&:= \xi^{(k)}  \cdot \nabla  u_t.
   \label{EconseqnsSig}
\end{align}
Versions of this model were previously considered in e.g. \cite{FG95,MR05} and discussed in \S 5.4 of \cite{F11}. 
We now verify pathwise energy conservation of the model \eqref{uEqn1} with  \eqref{Econseqnsf} and \eqref{EconseqnsSig}.  This property is most easily and directly established by using Stratonovich calculus and making use of well known properties of the $B(w,v)$ operator 
(see e.g. \cite{CF88}).   Since $\{\xi^{(k)}\}_{k\in \mathbb{N}}$ are divergence-free, we simply have
 \begin{align} 
 \frac{1}{2}\rmd \|u_t\|_{L^2}^2   &= -\Big( u_t, B\big(u_t \rmd t + \sum_k  \xi^{(k)} \circ \rmd W_t^{(k)} ,\, u_t\big)\Big)_{L^2(\Omega)} =0.
\end{align}
On the other hand, by Proposition \ref{propCirc} with $\sigma_t:=\sum_k \xi^{(k)} \cdot \nabla u_t$ and $f_t:=u_t\cdot \nabla u_t$, we have
\begin{align*}
\rmd \oint_{X_t(\Gamma)} u_t\cdot \rmd \ell  
&=   \sum_k \oint_{X_t(\Gamma)} 
\left(\frac{1}{2} \pounds_{\xi^{(k)}}^T(   \nabla \xi^{(k)}\cdot  u_t  ) \right)   \cdot \rmd \ell \   \rmd t+  \sum_k   \oint_{X_t(\Gamma)}  (\nabla \xi^{(k)}\cdot  u_t )\cdot \rmd \ell  \ \rmd W_t^{(k)}
\\&=
\sum_k   \oint_{X_t(\Gamma)}  (\nabla \xi^{(k)}\cdot  u_t) \cdot \rmd \ell  \circ \rmd W_t^{(k)}
.
\end{align*}
Thus, unless $\{\xi^{(k)}\}_{k\in \mathbb{N}}$ are spatially constant, the class of stochastic equations \eqref{uEqn1} with  \eqref{Econseqnsf} and \eqref{EconseqnsSig} which conserve energy pathwise are different with those that possess a pathwise Kelvin theorem.     
\end{rem}

\begin{rem}
Spatially constant noise coefficients $\{\xi^{(k)}\}_{k\in \mathbb{N}}$ define a privileged class of equations, solutions of which possess both circulation and energy conservation. 
In particular, when the $\{\xi^{(k)}\}_{k\in \mathbb{N}}$ are constants, the stochastic Euler--Poincar\'{e} equations are essentially deterministic Euler equations in disguise.  Specifically, let $u_t$ solve Eqn. \eqref{e1} and define
\be\label{vtou}
v_t(x):= u_t\left(x+ \sum_k \xi^{(k)} W^{(k)}_t\right).
\ee
  Then  the process  $v_t$ is incompressible $\nabla\cdot v_t=0$ and solves
\be\label{EvEqn}
\partial_t v_t + v_t\cdot \nabla v_t =-\nabla p_t ,
\ee
where $p_t$ solves the Poisson problem \eqref{presspoiss} to enforce incompressibility of the field $v_t$. Thus, formally, $v_t$ satisfies the usual deterministic Euler equation showing that these two equations have the same form.
To see this, suppose that a strong stochastic solution $u_t$ exists on $\Omega\times [0,T]$ (the existence of such a time $T$ is provide in \cite{CFH17}).  Using the It\^{o}--Wentzell formula in Stratonovich form \cite{K81} we obtain
\be\label{vevol}
\rmd v_t  = \left.\left(\rmd u_t + \sum_k \xi^{(k)}\cdot \nabla u_t \circ  \rmd W^{(k)}_t\right)\right|_{x+ \sum_k \xi^{(k)} {W^{(k)}_t}} \!\!\! \! = \rmd u_t\big|_{x+ \sum_k \xi^{(k)} {W^{(k)}_t}} + \sum_k \xi^{(k)}\cdot \nabla v_t \circ  \rmd W^{(k)}_t.
\ee
Now, our assumption of constant $\{\xi^{(k)}\}_{k\in \mathbb{N}}$ implies $  \mathbb{P}(\pounds_{\xi^{(k)}}^T u_t ) = \xi^{(k)} \cdot \nabla u_t$.  Thus,  using Eqn.  \eqref{e1} and  \eqref{vevol}, we obtain the equation,   $\rmd v_t + \mathbb{P}(v_t\cdot \nabla v_t)\rmd t =0$.  The classical time derivative $\partial_t v$ exists since $\mathbb{P}(v_t\cdot \nabla v_t)$ is continuous-in-time for each $x$ and, hence, Eqn. \eqref{EvEqn} follows.  Note that the change of variables above from $u_t$ to $v_t$ is not a Galilean transformation.

Since the transformation \eqref{vtou} is reversible, for sufficiently short times (while solutions exist), the unique stochastic solution $u_t$ of the SPDE \eqref{e1} can be recovered from the unique solution $v_t$ of deterministic Euler \eqref{EvEqn} by evaluating at a random spatial point
\be
u_t(x) = v_t \left(x- \sum_k \xi^{(k)} W^{(k)}_t\right).
\ee
Thus, as discussed in Chapter 5 of \cite{F11},  no regularizing effects can possibly come by adding this simple multiplicative noise to the Euler equations. If there is any non-trivial regularization-by-noise within the class of Euler-Poincar\'{e} models that we consider, it must arise due to spatial variation (and possibly solution dependence) of the noise correlates.

\end{rem}

Finally, we mention a related class of models in which the stochasticity is understood to arise from location uncertainty \cite{M14,M17a,M17b,M17c}.  These models also conserve energy pathwise but are distinct from all of those considered here.  In particular, they involve an additional division of the fields (velocity and pressure) into slow and fast fluctuating components and are obtained via a version of the Reynold's transport theorem.
\vspace{1mm}

\subsubsection*{Stochastic Navier-Stokes--Poincar\'{e} equations.}
In this note, we obtain also a class of natural stochastic generalizations of Navier-Stokes.  Similar to the Stochastic Euler--Poincar\'{e} equation, we ``randomize" the Navier-Stokes equations by insisting that they possess a certain analogue of the Kelvin theorem --  called the  Constantin--Iyer--Kelvin theorem --  which we now review.  In their paper \cite{CI08}, Constantin and Iyer proved that smooth solutions $u_t$ of the  Navier-Stokes equations
\begin{align}\label{ns1}
\partial_t u_t +(u_t\cdot  \nabla)  u_t&= -\nabla p_t +\nu \Delta u_t,\\
\nabla \cdot u_t&= 0, \\
u_t|_{t=0} &= u_0,\label{ns3}
\end{align}
are characterized by the following statistical Kelvin theorem; for all loops $\Gamma\subset \Omega$
\be\label{CIkelvin}
\int_\Gamma u_t \cdot \rmd \ell =\mathbb{E}\left[ \int_{A_t(\Gamma)} u_0\cdot \rmd \ell\right],
\ee
where $A_t:= X_t^{-1}$ is the back-to-labels map for the stochastic flow defined by the forward It\^o equation\footnote{Rather than introduce the back-to-labels map, the Constantin-Iyer Kelvin theorem can also be naturally stated in terms of time-reversed Brownian motion and backwards  It\^{o} SDEs \cite{DE17}.  For  detailed discussions of backward stochastic flows, see \cite{F06,K97}. }
\be\label{CITraj}
{\rmd} X_{t}(x)  = u_t(X_{t}(x)) \rmd t + \sqrt{2\nu} \ {\rmd} B_t, \qquad X_{0}(x)=x
\,.\ee
Here, $B_t$ is a $d$-dimensional standard Brownian motion.  The Constantin-Iyer Kelvin theorem has the
beautifully simple implication that smooth Navier-Stokes solutions are uniquely characterized as those velocity fields which have the property that circulations are \emph{backwards martingales} of the stochastic flow \eqref{CITraj}.

Unlike the pathwise Kelvin theorem \eqref{KelvinThm} which holds for solutions of the Stochastic Euler--Poincar\'{e} equations, \eqref{CIkelvin} is completely deterministic; since, the fluid velocity $u_t$ is a solution of equations \eqref{ns1}--\eqref{ns3}. 
The noise appearing in the flow \eqref{CITraj} is, in a sense, artificial. It plays a similar role as the noise used in Feynman-Kac representations for linear parabolic equations.  Namely, it is a mathematical tool to represent the Laplacian appearing in \eqref{ns1}.  However, unlike the Feynman-Kac representations  for linear equations, the stochastic Kelvin theorem \eqref{CIkelvin},\eqref{CITraj} constitutes a nonlinear fixed-point condition since the drift velocity in the trajectories \eqref{CITraj} is also the solution for which the circulation is computed \eqref{CIkelvin}.  In fact, a stochastic Weber formulation (equivalent to Kelvin theorem for smooth solutions) can be used to prove local existence of solutions of the incompressible Navier-Stokes \cite{I06}.  See also Remark \ref{stochWeber}, below.

We briefly recall some results connected to the formulation \eqref{CIkelvin}, \eqref{CITraj}. First, a different perspective on the Constantin-Iyer-Kelvin theorem was explored by Eyink in \cite{E10}, where it is shown that \eqref{CIkelvin} arises as a consequence of Noether's theorem via the particle relabelling symmetry of a certain stochastic action principle for the deterministic incompressible Navier-Stokes equations.  See also \cite{DE17} for a reformulation of Navier--Stokes as a system of stochastic Hamilton's equations, which yield a particularly simple derivation of the statistical Kelvin theorem.   This formulation has been since extended to domains with solid boundary \cite{CI11} and to a Riemannian manifold when the de Rham--Hodge Laplacian is the viscous dissipation operator \cite{FD18}.  Finally, Eyink \cite{E09} extended the work of Constantin and Iyer to nonideal hydromagnetic models.  There, a stochastic analogue of the classical Alfv\'{e}n theorem was proved to be equivalent to smooth solutions of the deterministic, nonideal, incompressible magnetohydrodynamic equations. \smallskip

 In what follows, we derive a class of SPDEs, smooth solutions of which possess (and are characterized by) a pathwise Constantin--Iyer Kelvin theorem. 
 We term these the \emph{stochastic Navier-Stokes--Poincar\'{e}} equations.
  Just as for \eqref{e1}--\eqref{e2}, these equations are driven by Brownian motions $\{W_t^{(k)}\}_{k\in \mathbb{N}}$ defined on the probability space $(\Xi,\mathcal{F}, {\bold P})$.  Relative to equations \eqref{e1}--\eqref{e2}, the stochastic Navier-Stokes--Poincar\'{e} equations contain additional terms which can be regarded as arising due to the presence of an artificial Brownian noise on the trajectories, just as in the Constantin-Iyer formalism.  This collection of 1-dimensional Brownian motions  $\{B_t^{(k)}\}_{k\in \mathbb{N}}$ is independent of the noise $\{W_t^{(k)}\}_{k\in \mathbb{N}}$.  We may now state our result.

 \begin{thm}[Characterization of Stochastic Navier-Stokes--Poincar\'{e} Fluids]\label{ThmKENS}
 
 Let $X_t$ be the flow defined by 
  \be\label{flowNS}
{\rmd}X_{t}(x) = u_t(X_{t}(x)) \rmd t +\sum_{k} \xi^{(k)}(X_{t}(x)) \circ  \rmd W_t^{(k)} + \sqrt{2\nu}\sum_{k} \eta^{(k)}(X_{t}(x))   \circ {\rmd} B_t^{(k)}  , \qquad X_{0}(x)=x,
 \ee
 for fixed smooth solenoidal vector fields $u_t: [0,T] \times \Omega \mapsto\mathbb{R}^d$ and   $\{\xi^{(k)}\}_{k\in \mathbb{N}}, \{\eta^{(k)}\}_{k\in \mathbb{N}}:\Omega\mapsto \mathbb{R}^d$.

Then, $u_t$ is a smooth solution of equations \eqref{uEqn1}  on $[0,T] \times \Omega$ with  
 \begin{align}\label{feqNS}
 f_t &= \pounds_{u_t}^T   u_t -\sum_k\frac{1}{2} \pounds_{\xi^{(k)}}^T( \pounds_{\xi^{(k)}}^T u_t) -\nu\sum_k \pounds_{\eta^{(k)}}^T( \pounds_{\eta^{(k)}}^T u_t)
 \,,\\
  \sigma_t^{(k)} &= \pounds_{\xi^{(k)}}^T u_t
  \,,
  \label{sigeqNS}
   \end{align}
   if and only if, for every rectifiable loop $\Gamma$, $u_t$ has the property that for $t\in[0,T]$, conditioned on realizations of $\{W^{(k)}\}_{k\in \mathbb{N}}$, circulations are backwards martingales
   \be\label{CIKelvin2}
   \oint_{\Gamma} u_t \cdot \rmd \ell =    \mathbb{E}\left[\oint_{A_{t}(\Gamma)} u_0 \cdot \rmd \ell\ \Bigg|   \ \mathcal{F}_{t}^{\{W^{(k)}\}}\right], \quad   \mathbf{P} \ a.s.
   \ee
  where $A_{t}:= X_t^{-1}$ is the back-to-labels map and $\mathcal{F}_{t}^{\{W^{(k)}\}}$ is the sigma-algebra generated by the increments $W^{(k)}_{s}-  W^{(k)}_{s'}$, $0\leq s'<s\leq t$, $k\in \mathbb{N}$.
 \end{thm}

 The idea above is that,  conditioning on the knowledge of the processes $\{W^{(k)}\}_{k\in \mathbb{N}}$ during $[0,t]$, we obtain a Constantin-Iyer-type circulation theorem \eqref{CIKelvin2} by averaging over the ``unresolved" Brownian motions $\{B^{(k)}\}_{k\in \mathbb{N}}$. The proof of Theorem \ref{ThmKENS} follows a different approach than that of Theorem \ref{ThmKE}. Instead of computing the rate of change of circulation and using the It\^o-Wentzell formula, we follow the original approach of \cite{CI08} and prove the equivalence of \eqref{uEqn1} with \eqref{feqNS} and \eqref{sigeqNS} with a fixed-point characterization in terms of a stochastic Weber formula.  This, in turn, is equivalent to the Kelvin theorem \eqref{CIKelvin2}.

\begin{rem}[Stochastic Fluids with Standard Viscous Friction]  \label{Remlaplacian}
If the $B^{(k)}$--noise amplitudes are constant and act only in the $d$ Euclidean directions $\{e_i\}_{i=1}^d$; that is, if 
\be\label{NSetas}
\{\eta^{(k)}\}_{k\in \mathbb{N}}= \{e_1, e_2, e_3, \dots, e_d, 0, 0,\cdots\},
\ee
 then $\nu\sum_k \pounds_{\eta^{(k)}}^T( \pounds_{\eta^{(k)}}^T u_t)= \nu \Delta u_t$ and the usual viscous Laplacian appearing in \eqref{ns1} is recovered.  Thus  \eqref{uEqn1} with \eqref{feqNS} and \eqref{sigeqNS} and $\{\eta^{(k)}\}_{k\in \mathbb{N}}$ given by \eqref{NSetas} form a family of stochastic generalizations of the deterministic Navier-Stokes which possess an exact analogue of the Constantin-Iyer Kelvin theorem  \eqref{CIKelvin2}.
\end{rem}

\begin{rem}[Energetic Properties of Circulation-Theorem Preserving Stochastic Fluids]\label{remEnergy2}
We now consider the energetics of the stochastic circulation--theorem--preserving models discussed here. Using  \eqref{uEqn1} with \eqref{feqNS} and \eqref{sigeqNS} (the case with \eqref{feq} and \eqref{sigeq} is obtained by setting $\nu\equiv 0$) we have by It\^{o}'s product rule  in the Hilbert space $L^2(\Omega)$ (see \cite{KR79})  that
\begin{align} 
\begin{split} 
\rmd \|u_t\|_{L^2(\Omega)}^2 &=\sum_k \left( u_t,  \frac{1}{2} \pounds_{\xi^{(k)}}^T( \pounds_{\xi^{(k)}}^T u_t) 
+ \nu \pounds_{\eta^{(k)}}^T( \pounds_{\eta^{(k)}}^T u_t) \right)_{L^2(\Omega)} \rmd t
+ \frac{1}{2} \int_\Omega\rmd \left[ u_t;u_t\right]_t\rmd x\\
&\qquad + \sum_k \left( u_t, \pounds_{\xi^{(k)}}^T u_t \,\rmd W_t^{(k)}\right)_{L^2(\Omega)}.
\end{split} 
\label{erg-circ-cons}
\end{align}
 Recall from \eqref{Lie-brkt} that the Lie derivative of a vector field $w$ is defined by $-\pounds_{\xi} w=[\xi, w]:= \xi\cdot \nabla w - w\cdot \nabla \xi$ and its adjoint operator satisfies the identity $( \pounds_{\xi}^T v, w)_{L^2(\Omega)}= - \langle v, \pounds_{\xi} w\rangle_{L^2(\Omega)}$, see  Eq. \eqref{adjointform}.
 Upon integrating by parts in \eqref{erg-circ-cons} using the adjoint relation and recalling that $u_t$ is divergence-free, we find
\begin{align} \nonumber
\rmd \|u_t\|_{L^2(\Omega)}^2 
&= -\frac{1}{2} \sum_k\left( \pounds_{\xi^{(k)}} u_t,\pounds_{\xi^{(k)}}^T u_t\right)_{L^2(\Omega)}  \rmd t - \nu   \sum_k \left(  \pounds_{\eta^{(k)}}u_t,  \pounds_{\eta^{(k)}}^T u_t \right)_{L^2(\Omega)} \rmd t\\
&\qquad + \frac{1}{2} \sum_k ( \mathbb{P} \pounds_{\xi^{(k)}}^T u_t, \pounds_{\xi^{(k)}}^T u_t)_{L^2(\Omega)}\rmd t + \sum_k (u_t,\nabla \xi^{(k)} \cdot u_t)_{L^2(\Omega)} \rmd W_t^{(k)}.\label{energyId}
\end{align}
Now, if $\xi$ and $v$ are divergence-free, then so is $\pounds_{\xi} v$. Consequently, we find that
\begin{align*} 
(  \pounds_{\xi^{(k)}} u_t,\pounds_{\xi^{(k)}}^T u_t)_{L^2(\Omega)}  -  ( \mathbb{P} \pounds_{\xi^{(k)}}^T u_t , \pounds_{\xi^{(k)}}^T u_t) _{L^2(\Omega)}&= (   \mathbb{P} (\pounds_{\xi^{(k)}} u_t- \pounds_{\xi^{(k)}}^T u_t), \pounds_{\xi^{(k)}}^T u_t)_{L^2(\Omega)}\\
&  =- (   \mathbb{P}(u_t\cdot \nabla \xi^{(k)}+  \nabla \xi^{(k)}\cdot  u_t),\pounds_{\xi^{(k)}}^T u_t) _{L^2(\Omega)}
\,.\end{align*}
Thus
\begin{align} \label{erg-noncons}
\begin{split}
\rmd \|u_t\|_{L^2(\Omega)}^2 
&= \frac{1}{2} \sum_k \left(   \mathbb{P}(u_t\cdot \nabla \xi^{(k)}+  \nabla \xi^{(k)}\cdot  u_t),\pounds_{\xi^{(k)}}^T u_t\right)_{L^2(\Omega)} \rmd t \\
&\qquad  - \nu   \sum_k  \left(   \pounds_{\eta^{(k)}}u_t,  \pounds_{\eta^{(k)}}^T u_t  \right)_{L^2(\Omega)}  \rmd t + \sum_k \left(u_t,\nabla \xi^{(k)} \cdot u_t\right)_{L^2(\Omega)}  \rmd W_t^{(k)}.
\end{split}
\end{align}
  Unlike equations \eqref{uEqn1} with  \eqref{Econseqnsf} and \eqref{EconseqnsSig} discussed in Remark \ref{remEnergy1}, the above computation shows that circulation-theorem preserving models do not, in general, satisfy a simple energy equality even when $\nu\equiv 0$ unless the $\xi^{(k)}$ are spatially constant.  Firstly, the energy in \eqref{erg-circ-cons} is a fluctuating quantity.  Moreover, even the average energy is neither increasing, nor decreasing, a priori.  Energy can be introduced or removed from the system by the action of spatial gradients of the noise correlates $\{ \xi^{(k)} \}$ on the solution.  However it is clear from \eqref{energyId} that if, for example, the $\eta^{(k)}$ are constant in space and span $\mathbb{R}^d$ (e.g. as in Remark \ref{Remlaplacian}) and if $\nu$ is taken to be sufficiently large, relative to the magnitude of $\xi^{(k)}$ and its spatial gradient,  then the system is dissipative on the average. That is, smooth solutions satisfy the inequality
\be
\mathbb{E} \frac{1}{2} \|u_t\|_{L^2(\Omega)}^2 \leq \frac{1}{2}\|u_0\|_{L^2(\Omega)}^2
\,,
\ee
where the expectation $\mathbb{E}$ denotes averaging over the Brownian motions $\{W_t^{(k)}\}_{k\in \mathbb{N}}$.
Thus, among the class of models  \eqref{uEqn1} satisfying \eqref{feqNS} and \eqref{sigeqNS} (i.e., among the choices for $\xi^{(k)}$), there are equations which have solutions possessing the Constantin-Iyer Kelvin theorem $\mathbf{P}$ almost surely and are, on the average, dissipative.  \smallskip
\end{rem}

\begin{rem}[Energetics of Dissipating Stochastic Fluids]\label{remEnergy3}
We describe one last class of models; those which dissipate energy pathwise and thus generalize \eqref{uEqn1} with  \eqref{Econseqnsf} and \eqref{EconseqnsSig} to the non-ideal setting.  Fixing solenoidal vector fields $\{\xi^{(k)}\}_{k\in \mathbb{N}}$ and $\{\eta^{(k)}\}_{k\in \mathbb{N}}$ and using the notation introduced for \eqref{stratencons}, they read
\be\label{stratenconsNs}
\rmd u_t + B(u_t\rmd t + \sum_k  \xi^{(k)}  \circ \rmd W_t^{(k)} , u_t)= \nu\sum_k \mathbb{P} \big({\eta}^{(k)} \cdot \nabla \mathbb{P} ({\eta}^{(k)} \cdot \nabla u_t)\big).
\ee
The form of the ``viscous term" is chosen as the piece of the double-(adjoint) Lie operator $\pounds_{\eta^{(k)}}^T( \pounds_{\eta^{(k)}}^T u_t)$ appearing in the  stochastic Navier-Stokes--Poincar\'{e} equations which ensures that this term cannot increase of energy. There are, of course, other choices for the dissipation operator.   In It\^{o} form, Eqn. \eqref{stratenconsNs} is  \eqref{uEqn1} with
 \begin{align}\label{EconseqnsNSf} 
 f_t&:= u_t \cdot \nabla u_t -   \sum_k \xi^{(k)}  \cdot \nabla \mathbb{P} ( \xi^{(k)}  \cdot \nabla   u_t) - \nu\sum_k {\eta}^{(k)} \cdot \nabla \mathbb{P} ({\eta}^{(k)} \cdot \nabla u_t), \\
   \sigma_t^{(k)}&:= \xi^{(k)}  \cdot \nabla  u_t.
   \label{EconseqnsNSSig}
\end{align}
Due to the properties of $B(w,v)$ discussed in Remark \ref{remEnergy1}, solutions of \eqref{uEqn1} with \eqref{EconseqnsNSf} and \eqref{EconseqnsNSSig} satisfy an pathwise energy balance 
\be
\frac{1}{2}\|u_t\|_{L^2(\Omega)}^2 = \frac{1}{2} \|u_0\|_{L^2(\Omega)}^2 - \nu \sum_k \int_0^t \|\mathbb{P}  \eta^{(k)} \cdot \nabla  u_s\|_{L^2(\Omega)}^2\rmd s 
\,, \quad    \mathbf{P} \ a.s.
\ee
Unsurprisingly, such fluids  do not possess a Constantin-Iyer Kelvin theorem, in general, unless the noise vector fields 
$\{ {\xi}^{(k)} \}_{k\in \mathbb{N}}$ and $\{ {\eta}^{(k)}\}_{k\in \mathbb{N}}$ are spatially constant.
\end{rem}

 \begin{rem}[Pathwise Stochastic Weber Formula] \label{stochWeber}
In the proof of Theorem  \ref{ThmKENS}, we show that solutions of the stochastic Navier-Stokes--Poincar\'{e} equations  \eqref{e1}--\eqref{e2}  satisfy a pathwise stochastic Weber formula:
\be\label{stochweber}
 u^\sharp_t(x)= \mathbb{E}\left[\mathbb{P} \Big[\nabla A_{t}(x)^T u_0(A_{t}(x))\Big]^\sharp\ \Big|\ \mathcal{F}_{t}^{\{W^{(k)}\}}\right], \quad   \mathbf{P} \ a.s.
\ee
in which  the expectation averages over the standard Brownian motions $\{B^{(k)}\}_{k\in \mathbb{N}}$. 
By Stokes theorem applied to  \eqref{CIKelvin2}, we find that the vorticity-flux through comoving areas is statistically frozen 
   \be
   \int\!\!\!  \int_{S} \omega_t \cdot \rmd S =    \mathbb{E}\left[ \int\!\!\!  \int_{A_{t}(S)} \omega_0 \cdot \rmd S\ \Bigg|\ \mathcal{F}_{t}^{\{W^{(k)}\}}\right], \quad    \mathbf{P} \ a.s.
   \ee
 \end{rem}

 \section{Proofs}

 \begin{proof}[Proof of Proposition \ref{propCirc}]

The proof follows from a direct computation.  
 First, we convert \eqref{flow2} to an equivalent  It\^o SDE governing the paths
   \be\label{itopaths}
 \rmd X_t(x) = \left(u_t+ \frac{1}{2}\sum_k\xi^{(k)}\cdot \nabla \xi^{(k)} \right)\Bigg|_{X_t(x)} \rmd t +\sum_k \xi^{(k)}(X_t(x))   \rmd W_t^{(k)}, \qquad X_0(x)=x.
 \ee 
  The new term appearing in \eqref{itopaths} is called the ``noise-induced drift".
  Now, for any rectifiable loop $\Gamma$, let $\Gamma(s): [0,1]\mapsto \Gamma$ be a parametrization.  Then the circulation around the loop $\Gamma$ can be represented as
 \be
 \oint_{X_t(\Gamma)} u_t\cdot \rmd \ell =  \int_0^1  \frac{d}{ds} X_t(\Gamma(s))  \cdot u_t(X_t(\Gamma(s))) \rmd s=  \int_0^1   \Gamma'(s) \cdot \nabla X_t(\Gamma(s)) \cdot  u_t(X_t(\Gamma(s)))  \rmd s.  
 \ee
Upon differentiating the circulation in this representation and applying the It\^{o} product rule, we have
 \be\label{itoprod}
 \rmd  \oint_{X_t(\Gamma)} u_t\cdot \rmd \ell  =   \int_0^1 \Gamma'(s)\cdot \Big(\nabla X_t\cdot \rmd u_t(X_t)   +\rmd\nabla X_t\cdot   u_t(X_t)    + \rmd \left[  \nabla X_t;u_t(X_t)\right]_t\Big)\Big|_{x=\Gamma(s)} \rmd s.
 \ee
 The flow $u_t$ is random, driven by the same noise as on the particle trajectories.  Therefore,  to compute the stochastic differential $  \rmd (u_t(X_t(x)))$, we  apply the { It\^o-Wentzell formula}. For details, see, e.g., Theorem 1.1. of \cite{K81} or Theorem 3.3.1 of \cite{K97}.  This calculation introduces the Wentzell correction, as 
  \begin{align*}
  \rmd (u_t(X_t(x))) &= (\rmd u_t +\rmd X_t\cdot \nabla u_t)\big|_{X_t(x)} + \frac{1}{2} \nabla \otimes \nabla u_t  : \rmd \left[ X_t, X_t \right]_t + \rmd \left[\nabla u_t; X_t\right]_t  \big|_{X_t(x)}\\
  &= \Big(\rmd u_t +\big( u_t\cdot \nabla u_t +\frac{1}{2}\sum_k(\xi^{(k)}\cdot \nabla) \xi^{(k)}\cdot \nabla u_t  + \frac{1}{2}\sum_k \xi^{(k)} \otimes \xi^{(k)} : \nabla \otimes \nabla u_t\big)\rmd t\Big)\Big|_{X_t(x)}\\
  &\qquad   +\sum_k (\xi^{(k)} \cdot \nabla) u_t \big|_{X_t(x)}  \rmd W_t^{(k)}  + \rmd \left[\nabla u_t; X_t\right]_t  \big|_{X_t(x)}.
      \end{align*}
  To compute the Wentzell correction ${\rmd \left[\nabla u_t; X_t\right]_t  \big|_{X_t(x)}}$ explicitly, we take the gradient of Eq.  \eqref{uEqn1}
   \begin{align}\label{gradientEq}
 \rmd \nabla u_t =- \nabla \mathbb{P} f_t \rmd t -\sum_{k}  \nabla\mathbb{P}  \, \sigma_t^{(k)}  \rmd W_t^{(k)}.
\end{align}
The martingale part of $  \rmd \nabla u_t  $ is $-\sum_{k}  \nabla\mathbb{P}   \sigma_t^{(k)}  \rmd W_t^{(k)}$. Consequently, the Wentzell correction is given by
\begin{align}
\rmd \left[\nabla u_t; X_t\right]_t  \big|_{X_t(x)}:=\rmd \left[\partial_i u_t,X_t^i\right]_t \big|_{X_t(x)}&=-\sum_{k}     (\xi^{(k)} \cdot \nabla) \mathbb{P} \sigma_t^{(k)} \big|_{X_t} \rmd t.
\end{align}
Putting this together, we obtain the full differential
   \begin{align}\nonumber
  \rmd (u_t(X_t(x)))   &= \Big(u_t\cdot \nabla u_t -\mathbb{P}f_t\Big)\Big|_{X_t(x)} \rmd t+\sum_k \Big((\xi^{(k)} \cdot \nabla) u_t -\mathbb{P}\sigma_t\Big) \big|_{X_t(x)}  \rmd W_t^{(k)} \\
  &\qquad + \sum_{k}\Big(\frac{1}{2}(\xi^{(k)}\cdot \nabla) \xi^{(k)}\cdot \nabla u_t  + \frac{1}{2}\xi^{(k)} \otimes \xi^{(k)} : \nabla \otimes \nabla u_t- { \xi^{(k)} \cdot \nabla \mathbb{P} \sigma_t^{(k)}}\Big)\Big|_{X_t(x)}\rmd t.\label{itoWentzellU}
  \end{align}
Next, the gradient of the stochastic flow is easily found to satisfy
  \begin{align}\label{gradflow}
 \rmd \nabla X_t(x) &= \nabla X_t(x)\cdot\left( \nabla u_t(X_t(x)) +\frac{1}{2} \sum_k\nabla( \xi^{(k)}\cdot \nabla \xi^{(k)})   \right)\rmd t 
 \\&\hspace{40mm} + \sum_k \nabla X_t(x) \cdot \nabla \xi^{(k)}(X_t(x))  \rmd W_t^{(k)}, \nonumber \\
  \nabla X_0(x) 
  &=\mathbb{I}. 
 \end{align}
In view of \eqref{itoWentzellU} and \eqref{gradflow},  the quadratic cross-variation between the Lagrangian velocity and deformation matrix is
  \be
  \rmd \left[ \nabla X_t;u_t(X_t) \right]_t= \nabla X_t\cdot \sum_k \nabla \xi^{(k)} \cdot\Big( (\xi^{(k)} \cdot \nabla) u_t-\mathbb{P}\sigma_t\Big)\big|_{X_t(x)}\rmd t .
  \ee
Finally, the remaining term in \eqref{itoprod} can be expressed using \eqref{gradflow} as follows
  \be
  \rmd\nabla X_t \cdot u_t(X_t)   =  \nabla X_t \cdot \left(\left( \nabla \left(\frac{1}{2} |u_t|^2 \right)+ \frac{1}{2} \sum_k\nabla( \xi^{(k)}\cdot \nabla \xi^{(k)})\cdot u_t   \right)\rmd t  +\sum_k  (\nabla \xi^{(k)} \cdot u_t) \rmd W_t^{(k)}\right)\Bigg|_{X_t(x)}.
  \ee
  Upon putting the various elements of this computation this together, we have
  \begin{align}\nonumber
 \rmd  &\oint_{X_t(\Gamma)} u_t\cdot \rmd \ell  =   \int_0^1\Gamma'(s)\cdot \nabla X_t(\Gamma(s))\cdot \left( u_t\cdot \nabla u_t + \nabla \left(\frac{1}{2} |u_t|^2\right) - \mathbb{P}f_t\right)\Big|_{X_t(\Gamma(s))}\rmd t\rmd s\\ \nonumber
  &\qquad+   \sum_{k} \int_0^1\Gamma'(s)\cdot \nabla X_t(\Gamma(s))\cdot   \Big( \frac{1}{2}(\xi^{(k)}\cdot \nabla) \xi^{(k)}\cdot \nabla u_t  + \frac{1}{2}\xi^{(k)} \otimes \xi^{(k)} : \nabla \otimes \nabla u_t\\ \nonumber
  &\hspace{40mm} +  \frac{1}{2} \nabla( \xi^{(k)}\cdot \nabla \xi^{(k)})\cdot u_t  +   \nabla \xi^{(k)} \cdot (\xi^{(k)} \cdot \nabla) u_t- \pounds_{\xi^{(k)}}^T  \mathbb{P} \sigma_t^{(k)} \Big)\Big|_{X_t(\Gamma(s))} \rmd s\\
  &\qquad +   \sum_{k} \int_0^1 \Gamma'(s)\cdot \nabla X_t(\Gamma(s))\cdot  \left(\xi^{(k)} \cdot \nabla u_t    +\nabla \xi^{(k)} \cdot u_t  - \mathbb{P}\sigma_t^{(k)}\right)\Big|_{X_t(\Gamma(s))} \rmd W_t^{(k)}  \rmd t \rmd s.\label{comput}
    \end{align}
    Recall from the computation \eqref{vanishingOnGrad} that $\pounds_{\xi}^T\mathbb{P} v= \pounds_{\xi}^T v + \nabla (\xi\cdot \nabla q),$ for any vector field $v$ and some scalar function ${q}$.  
Since gradients vanish upon integration over closed loops (and, consequently, the action of the Leray projector is trivial on loop integrals), we have that
    \begin{align}\nonumber
   \rmd  &\oint_{X_t(\Gamma)} u_t\cdot \rmd \ell =  \oint_{X_t(\Gamma)}  \left(\pounds_{u_t}^T   u_t- f_t\right) \rmd t  \cdot \rmd \ell \\ \nonumber
   &\qquad +  \sum_k \oint_{X_t(\Gamma)} \left(\frac{1}{2}(\xi^{(k)}\cdot \nabla) \xi^{(k)}\cdot \nabla u_t +  \frac{1}{2}\xi^{(k)} \otimes \xi^{(k)} : \nabla \otimes \nabla u_t +  \frac{1}{2} \nabla( \xi^{(k)}\cdot \nabla \xi^{(k)})\cdot u_t\right.\\ 
   &\qquad%\hspace{90mm}
   \left.
   +\   \nabla \xi^{(k)} \cdot (\xi^{(k)} \cdot \nabla) u_t-  \pounds_{\xi^{(k)}}^T \sigma_t^{(k)}\right)  \rmd t  \cdot \rmd \ell \
   + \sum_k   \oint_{X_t(\Gamma)} \left( \pounds_{\xi^{(k)}}^T u_t  -\sigma_t\right) \rmd W_t^{(k)}\cdot \rmd \ell. \label{nexttolastcirc}
  \end{align}
Now note that the double (adjoint) Lie derivative \eqref{e2} can be expanded as follows:
\begin{align}\nonumber
\pounds_{\xi}^T(  \pounds_{\xi}^T v) &= \pounds_{\xi}^T(  \xi^j \partial_j  v_i+ \partial_i \xi^j  v_j) \\ \nonumber
&= \xi^k\partial_k(  \xi^j \partial_j  v_i+ \partial_i \xi^j  v_j) + \partial_i \xi^k(  \xi^j \partial_j  v_k+ \partial_k \xi^j  v_j)\\ \nonumber
&= (\xi\cdot \nabla)\xi\cdot \nabla  v_i+  (\xi \otimes \xi): (\nabla \otimes \nabla) v\\ \nonumber
&\quad + \partial_i \xi^j  (\xi\cdot \nabla)v_j+ ((\xi\cdot \nabla)\partial_i \xi^j) v_j  + (\partial_i \xi^j)(\xi\cdot \nabla)  v_j+ (\partial_i \xi^k)\partial_k \xi^j  v_j\\   
&= (\xi\cdot \nabla)\xi\cdot \nabla  v+ (\xi \otimes \xi): (\nabla \otimes \nabla) v+2 \nabla  \xi\cdot   (\xi\cdot \nabla)v+ \nabla ((\xi\cdot \nabla) \xi) \cdot v.  \label{lieComps}
\end{align}
Upon substituting this simplification into \eqref{nexttolastcirc}, we finally obtain equation \eqref{circulations}.

We remark that, geometrically, the right hand side of Eqn. \eqref{lieComps} is the $L^2$ dual of the double Lie bracket $[\xi,[\xi,u]\,]$ of the vector field $\xi$ acting on $u$; that is,
\[
\big\langle \pounds_{\xi}^T(\pounds_{\xi}^T v),u \big\rangle 
= 
\big\langle  v, [\xi,[\xi,u]\,]  \big\rangle.
\]
Moreover, in three-dimensional Euclidean space, by using the second form given in Eqn. \eqref{e2} one can obtain the following alternative expression for the double (adjoint) Lie derivative involving cross-products and the curl operator:
\begin{align} \nonumber
\pounds_{\xi}^T(  \pounds_{\xi}^T v)  &= \pounds_{\xi}^T\big( - \xi \times \curl v + \nabla (\xi\cdot v)\big) %= \pounds_{\xi}^T(-  \xi \times \curl v) +  \pounds_{\xi}^T( \nabla (\xi\cdot v))
\\ \nonumber
  &= \xi\times \curl(  \xi \times \curl v) - \nabla (\xi\cdot   ( \xi \times \curl v))  +  \pounds_{\xi}^T( \nabla (\xi\cdot v))\\ \nonumber
  &= \xi\times \curl(  \xi \times \curl v)  +  \pounds_{\xi}^T( \nabla (\xi\cdot v))\\
    &= \xi\times \curl(  \xi \times \curl v)+   \nabla (\xi\cdot \nabla (\xi\cdot v)), \label{DLieCross}
\end{align}
where we have used the identity $\xi\cdot   ( \xi \times \curl v) =\curl v \cdot   ( \xi \times \xi)=0  $ and the fact that $\pounds_{\xi}^T\nabla q= \nabla (\xi\cdot \nabla q)$ which was verified in Eqn. \eqref{vanishingOnGrad}.  Note that the final term in \eqref{DLieCross} is a total gradient and therefore vanishes upon integration over any closed, rectifiable loop $\Gamma$.
 \end{proof}

 \begin{proof}[Proof of Theorem \ref{ThmKE}]$\,$
  
\noindent  We proceed in the same spirit as in the proof of the equivalence of the usual Kelvin theorem to smooth solutions of deterministic Euler given in the Introduction.
\vspace{2mm}

\noindent \textbf{Direction 1: Stochastic Euler-Poincar\'{e} solutions have a pathwise Kelvin Theorem.}
In view of Proposition \ref{propCirc}, one direction is simple: by using equation \eqref{uEqn1} with $f_t$ and $\{\sigma_t^{(k)}\}_{k\in \mathbb{N}}$  defined by  \eqref{feq} and \eqref{sigeq} in Theorem \ref{ThmKE}, and applying Proposition  \ref{propCirc} to the unique smooth solution $u_t$ for given initial conditions $u_0$ (which always exists provided, at least, that $T$ is taken sufficiently small  \cite{CFH17,F18}, see Remark \ref{localExRem}), one has that realization-by-realization of the Brownian processes $\{W_t^{(k)}\}_{k\in\mathbb{N}}$ circulations are materially conserved \eqref{KelvinThm}.  
\vspace{0mm}

\noindent \textbf{Direction 2: Pathwise Kelvin Theorem for all loops implies $u_t$ is a stochastic  Euler-Poincar\'{e}  solution.}
 For the other direction, assume that the circulation is conserved along all material loops $\Gamma$.  Since $u_t$ and $\{\xi^{(k)}\}_{k\in \mathbb{N}}$ are assumed smooth, the map $x\mapsto X_t(x)$ is a $\mathcal{F}_t$--adapted diffeomorphism \cite{K97,LJW84}. Its spatial inverse $A_t= X_t^{-1}$ is  $\mathcal{F}_t$--adapted, pointwise in $x$.  See Remark \ref{remReg} for a precise, sufficient regularity assumption. First we establish the form of the noise in the SPDE.
 \vspace{2mm}

\noindent \textbf{Form of noise:}. From  \eqref{circulations}, the quadratic variation of the circulation (denoted for a process $\zeta_t$ by $[\zeta_t ]_t$) is
\be
 \bigg[  \oint_{X_t(\Gamma)} u_t\cdot \rmd \ell \bigg|_{t=0}^{t={T'}} \bigg]_{T'}
 =  \sum_k \int_0^{T'}\int_0^1  \left|\Gamma'(s) \cdot \nabla X_t(\Gamma(s)) 
 \cdot   \left( \pounds_{\xi^{(k)}}^T u_t - \sigma_t^{(k)} \right)\big|_{X_t(\Gamma(s))}\right|^2 \rmd t \rmd s\,,
\ee
for any $T'\in[0,T]$.
On the other hand, if the pathwise Kelvin theorem holds, then the left-hand-side must vanish.  By assumption, the function $f(t,s):= \Gamma'(s) \cdot \nabla X_t(\Gamma(s)) \cdot   \left( \pounds_{\xi^{(k)}}^T u_t - \sigma_t^{(k)} \right)\big|_{X_t(\Gamma(s))}$ is continuous on $[0,T]\times [0,1]$.  Thus, we conclude that  for all $(t,s)\in [0,T]\times [0,1]$,  
\be\label{vanishing1}
\Gamma'(s) \cdot \nabla X_t(\Gamma(s)) \cdot   \left( \pounds_{\xi^{(k)}}^T u_t - \sigma_t^{(k)}  \right)\big|_{X_t(\Gamma(s))}\rmd t=0, \qquad \forall\  k\in \mathbb{N}.
\ee
We now show that the matrix  $\nabla X_t$ in \eqref{vanishing1} is non-singular almost surely for all  $x\in \Omega$.  For this, we apply
 \begin{lemma}\label{jacobian}
Fix smooth vector fields $b_t: [0,T] \times \Omega \mapsto\mathbb{R}^d$ and  $\{\xi^{(k)}\}_{k\in \mathbb{N}}:\Omega\mapsto \mathbb{R}^d$.  Let $x\mapsto X_{s,t}(x)$  be the regular stochastic flow of diffeomorphisms \cite{K97} associated to the It\^{o} SDE
\be\label{flow_jac}
\rmd X_{t}(x) = b_t(X_{t}(x)) \rmd t +\sum_k \xi^{(k)}(X_{t}(x)) \rmd W_t^{(k)}, \qquad X_{0}(x)=x.
\ee
 Then the following formula for the Jacobian holds
\be\label{detform}
\det(\nabla X_t(x)) = \exp\left(\int_0^t \left(\nabla \cdot b_t  - \frac{1}{2} \sum_k (\nabla \xi^{(k)})^T:\nabla \xi^{(k)} \right)\Bigg|_{X_{s}(x)} \rmd s + \sum_k\int_0^t \nabla \cdot \xi^{(k)}\big|_{X_{s}(x)}\rmd W_s^{(k)}\right).
\ee
\end{lemma}
\begin{proof}
Recall the classic formula $\ln (\det A) = {\rm tr} (\ln A )$, for any invertible matrix $A$. The first and second order Gateaux derivative of $\ln(\det A)$ in direction $\phi$  and in $(\phi,\psi)$ resp. may then be computed to be 
 \be\label{gatDerv}
D\ln(\det A)[\phi] = \tr[\phi A^{-1}],\qquad  D^2\ln( \det A)(A)[\phi,\psi] =-\tr[\phi A^{-1} \psi A^{-1}].
 \ee
The proof will follow as a direct computation. First, by taking the gradient in the initial data of \eqref{flow_jac}  we have
   \begin{align}\label{gradflowjac}
 \rmd \nabla X_t(x) = \nabla X_t(x)\cdot \nabla b_t(X_t(x)) \rmd t + \sum_k \nabla X_t(x) \cdot \nabla \xi^{(k)}(X_t(x))  \rmd W_t^{(k)}, \quad \nabla X_0(x)=\mathbb{I}.
 \end{align}
 Next, applying It\^{o}'s formula and using equation \eqref{gradflowjac} and the formulae \eqref{gatDerv}, we compute
 \begin{align*} 
 \rmd \ln\det(\nabla X_t(x))  &=D\ln\det(\nabla X_t(x))[  \nabla X_t(x)\cdot \nabla b_t(X_t(x))]\rmd t \\
 &\quad + \sum_kD\ln\det(\nabla X_t(x))\left[  \nabla X_t(x) \cdot \nabla \xi^{(k)}(X_t(x)) \right]\rmd W_t^{(k)}\\
 &\quad + \frac{1}{2}  \sum_kD^2\ln\det(\nabla X_t(x))\left[ \nabla X_t(x) \cdot \nabla \xi^{(k)}(X_t(x)),  \nabla X_t(x) \cdot \nabla \xi^{(k)}(X_t(x)) \right]\rmd t\\
 &= \tr(\nabla b_t)(X_t(x))\rmd t + \sum_k \tr(\nabla \xi^{(k)})(X_t(x))\rmd W_t^{(k)} 
 - \frac{1}{2}\sum_k (\nabla \xi^{(k)})^T:  \nabla \xi^{(k)}\big|_{X_t(x)} \rmd t.
 \end{align*}
We integrate in time and evaluate $\ln\det(\nabla X_0(x))=0$, since $\det(\nabla X_0(x))=1$. This yields a fomula for $\ln\det(\nabla X_t(x))$; whereupon formula \eqref{detform} emerges, upon exponentiating the result.
\end{proof}
In view of \eqref{itopaths}, we apply Lemma \ref{jacobian} with $b_t= u_t+ \frac{1}{2}\sum_k\xi^{(k)}\cdot \nabla \xi^{(k)} $.  Note that 
$$
\nabla \cdot b_t=\nabla \cdot u_t+  \frac{1}{2} \sum_k \left((\nabla \xi^{(k)})^T:\nabla \xi^{(k)} + \sum_k\xi^{(k)}\cdot \nabla (\nabla \cdot \xi^{(k)})\right).
$$
Thus, for divergence-free vector fields $u_t$ and $\{\xi^{(k)}\}_{k\in \mathbb{N}}$, we find from \eqref{detform} that the Stratonovich stochastic flow \eqref{flow} is volume preserving,
$\det(\nabla X_t(x)) = 1$. Thus, the kernel of $\nabla X_t$ is trivial $\mathbf{P}$ almost surely pointwise in  $(t,x)\in [0,T]\times \Omega$. Now, for any point $x\in \Omega$, choose a collection of loops $\{\Gamma_i(s)\}_{i=1,\dots,d}$ such that  at $x=\Gamma_i(s_i)$ for some $s_i\in[0,1]$ and with linearly independent tangents $\{\Gamma_i'(s_i)\}_{i=1,\dots,d}$.  Since \eqref{vanishing1} holds for all such loops and the matrix $\nabla X_t(x)$ is non-singular, it follows that $\pounds_{\xi^{(k)}}^T u_t - \sigma_t^{(k)} =0$ at $X_t(x)$ for all $t\in [0,T]$, $\mathbf{P}$ almost surely.  The above argument can be applied to all $x\in \Omega$ by choosing the appropriate collection of loops $\{ \Gamma_i\}$ and we conclude,
\be
\sigma_t^{(k)}|_{X_t(x)}=  \pounds_{\xi^{(k)}}^T u_t|_{X_t(x)},  \quad  \forall\  k\in \mathbb{N}, \ \ (t,x)\in [0,T]\times\Omega, \quad \mathbf{P}\  a.s.
\ee
Finally, fix any $y\in \Omega$.  Then, for any $t\in [0,T]$ and $\mathbf{P}$ a.e. 
$\varpi$ (where $\varpi$ denotes sample space dependence), letting $x= A_t^\varpi(y)$ allows us to conclude that $\sigma_t^{(k)}=  \pounds_{\xi^{(k)}}^Tu_t$ for all $(t,y)\in  [0,T]\times \Omega$,  $\mathbf{P}$ almost surely.

\noindent \textbf{Form of drift.} Upon using the fact that $ \sigma_t^{(k)} = \pounds_{\xi^{(k)}}^T u_t$,  pointwise in spacetime $\mathbf{P}$ a.s., Prop. \ref{propCirc}, implies that
  \begin{align*}
  \int_0^{T'}\oint_{X_t(\Gamma)}  & \left( \pounds_{u_t}^T   u_t-\sum_k\frac{1}{2} \pounds_{\xi^{(k)}}^T( \pounds_{\xi^{(k)}}^T u_t)  -f_t\right) \cdot \rmd \ell \ \rmd t  = 0\,,
  \end{align*}
  for all rectifiable loops $\Gamma$ and all ${T'}\in[0,T]$.  Since it is continuous, the integrand in the time integral above must vanish identically for all $t\in [0,T]$. Now, let $\Gamma'$ be any rectifiable loop.  Then, for any fixed $t\in [0,T]$ and $\mathbf{P}$ a.e. $\varpi$, let $\Gamma=A_t^\varpi(\Gamma')$.   Thus, we deduce that for any loop $\Gamma'$ the following holds
  $$
\oint_{\Gamma'}  \left( \pounds_{u_t}^T   u_t-\sum_k\frac{1}{2} \pounds_{\xi^{(k)}}^T( \pounds_{\xi^{(k)}}^T u_t)  -f_t\right) \cdot \rmd \ell   = 0.
 $$
  Finally, we can conclude that there exists a scalar process $q_t$ (not necessarily of bounded variation) such that
  \be\nonumber
f_t  = \left( \pounds_{u_t}^T   u_t -\sum_k\frac{1}{2} \pounds_{\xi^{(k)}}^T( \pounds_{\xi^{(k)}}^T u_t) \right)+ \nabla  q_t \,.
  \ee
%The incompressibility constraint uniquely fixes $q_t=p_t$ up to additive constant, where $p_t$ is The pressure $p_t$ is  the unique periodic solution (up to additive constant) of the following elliptic problem
%$$
%-\Delta   p_t\rmd t = (\nabla\otimes \nabla):(u_t\otimes u_t)  \rmd t+ \sum_{k}  \nabla \cdot  \pounds_{\xi^{(k)}}^T u_t \circ \rmd W_t^{(k)}.
%$$
By the fact that the Leray--Hodge projector $\mathbb{P}$ vanishes on gradients, it follows that  Eqn. \eqref{uEqn1} is satisfied with $f_t$ given by the expression \eqref{feq}. 
 \end{proof}

  \begin{proof}[Proof of Theorem \ref{ThmKENS}] 
Our proof employs a different method than that of  our Theorem \ref{ThmKE}.  In particular, we establish equivalence to a stochastic Weber formula, 
\be\label{stochweber1}
 u_t(x)= \mathbb{E}\left[\mathbb{P} (\nabla A_{t}(x))^T u_0(A_{t}(x))\ \Big|\ \mathcal{F}_{t}^{\{W^{(k)}\}}\right], \quad   \mathbf{P} \ a.s.
\ee
where $A_{t}= X_{t}^{-1}$ is the back-to-labels map and $X_t$ solves \eqref{flowNS}.  Note that, together, equations \eqref{stochweber1} and  \eqref{flowNS} form a fixed point problem.  It should be possible to solve this problem (pathwise in $W^{(k)}$) by combining the methods of  \cite{F18} for the stochastic Euler--Poincar\'{e}  with those of  \cite{I06} which establish local existence of deterministic Navier-Stokes from the stochastic Weber formula.  We do not pursue this issue here. Instead, we simply assume that smooth solutions of  \eqref{flowNS}, \eqref{stochweber1} exist, at least for sufficiently small times $T>0$. 

Note that it is clear that for sufficiently smooth $u_t$, the stochastic Weber formula \eqref{stochweber1} is equivalent to its integrated form on loops - the Constantin-Iyer Kelvin theorem:
\begin{align*}
 \oint_\Gamma u_t\cdot \rmd \ell &= \mathbb{E}\left[ \oint_\Gamma  \mathbb{P} (\nabla A_{t}(x))^T u_0(A_{t}(x))\ \Big|\ \mathcal{F}_{t}^{\{W^{(k)}\}}\right] \cdot \rmd \ell = \mathbb{E}\left[ \oint_{ A_{t}(\Gamma)}    u_0 \cdot \rmd \ell \Big|\ \mathcal{F}_{t}^{\{W^{(k)}\}}\right] .
\end{align*}
Thus, equivalence to the Constantin-Iyer Kelvin theorem for smooth solutions will follow from the same fixed point problem and the stochastic Navier-Stokes--Poincar\'{e} equations 
 \eqref{uEqn1} with $f_t$ and $\{\sigma_t^{(k)}\}_{k\in \mathbb{N}}$  defined by  \eqref{feqNS} and \eqref{sigeqNS}.   We note that this strategy has also been used in \cite{E09} to prove the equivalence of certain non-ideal hydromagnetic models to their stochastic Alfv\'{e}n theorems. 
 \vspace{2mm}

\noindent \textbf{Direction 1: Solution of the fixed-point problem \eqref{flowNS}, \eqref{stochweber1} solves Eq.  \eqref{uEqn1}.}
We first prove that a solution of the fixed point problem \eqref{stochweber1} provides a representation for a solution of  Eqn. \eqref{uEqn1} with  \eqref{feqNS} and \eqref{sigeqNS}.  We begin by using \eqref{stochweber1} to show that for any solenoidal vector field $v$ for all $0\leq s\leq t\leq T$ we have 
\begin{align}\nonumber
\langle u_t, v\rangle_{L^2} &= \mathbb{E}\left[ \langle  (\nabla A_{t})^T u_0(A_{t}), v \rangle_{L^2}\  \Big| \ \mathcal{F}_{t}^{\{W^{(k)}\}}\right]= \mathbb{E}\left[ \langle   u_0(A_{t}), (\nabla A_{t}) v \rangle_{L^2}\  \Big| \ \mathcal{F}_{t}^{\{W^{(k)}\}}\right]\\
&= \mathbb{E}\left[ \langle   u_0, (\nabla A_{t})(X_t) v(X_{t}) \rangle_{L^2}\  \Big| \ \mathcal{F}_{t}^{\{W^{(k)}\}}\right]= \mathbb{E}\left[ \langle   u_0, (A_{t})^* v \rangle_{L^2}\  \Big| \ \mathcal{F}_{t}^{\{W^{(k)}\}}\right]\label{newpaireqn}
\end{align}
where we have recalled  that $(\nabla A_{t})(X_{t}) v(X_{t}):= (A_{t})^* v= (X_{t}^{-1})^* v$ is the pull-back of $v$ by the flow $X_{t}$.   Now, by Kunita's formula \cite{K97}, we have for flows $X_{t}$ generated by the SDE  \eqref{flowNS} that 
\begin{align}\nonumber
(A_{t})^* v &= v + \sum_k \int_0^t (A_{s})^*   \pounds_{\xi^{(k)}} v\  \rmd W_s^{(k)} + \sqrt{2\nu} \sum_k \int_0^t (A_{s})^*   \pounds_{\eta^{(k)}} v \ \rmd B_s^{(k)} \\
&\quad + \int_0^t \left[ (A_{s})^*  \pounds_{u_s} v + \frac{1}{2} \sum_k  (A_{s})^*   \pounds_{\xi^{(k)}}( \pounds_{\xi^{(k)}} v) + \nu \sum_k  (A_{s})^*   \pounds_{\eta^{(k)}}( \pounds_{\eta^{(k)}} v) \right] \rmd s.\label{KunflowEqn}
\end{align}
In the interest of being self-contained, we prove the identity \eqref{KunflowEqn} in a slightly different but equivalent form in Lemma \ref{wLem} below.  Substituting \eqref{KunflowEqn} into \eqref{newpaireqn} and recalling that $ \sqrt{2\nu} \sum_k \int_0^t (A_{s})_*   \pounds_{\eta^{(k)}} v \ \rmd B_s^{(k)}$ is a martingale, conditioned on the history of the process $W_t^{(k)} $, we have
\begin{align}\nonumber
\langle u_t, v\rangle_{L^2} &= \langle u_0, v\rangle_{L^2}  + \sum_k\mathbb{E}\left[ \int_0^t  \langle u_0, (A_{s})^*   \pounds_{\xi^{(k)}} v\rangle_{L^2(\Omega)}\  \rmd W_s^{(k)} \  \Big| \ \mathcal{F}_{t}^{\{W^{(k)}\}}\right]
 \\ \nonumber
&\quad + \int_0^t \mathbb{E}\left[  \langle u_0, (A_{s})^*  \pounds_{u_s} v\rangle_{L^2(\Omega)} + \frac{1}{2} \sum_k   \langle u_s, (A_{s})^*   \pounds_{\xi^{(k)}}( \pounds_{\xi^{(k)}} v)\rangle_{L^2(\Omega)}\right.\\
&\qquad\qquad\qquad\qquad\qquad  \left.+ \nu \sum_k   \langle u_0, (A_{s})^*   \pounds_{\eta^{(k)}}( \pounds_{\eta^{(k)}} v)\rangle_{L^2(\Omega)}\  \Big| \ \mathcal{F}_{t}^{\{W^{(k)}\}}\right]  \rmd s.
\end{align}
Upon using the equivalence \eqref{newpaireqn}, which holds for any divergence-free vector field (a property which is satisfied by all of $\pounds_{u_t} v$, $\pounds_{\xi^{(k)}} v$,  $\pounds_{\xi^{(k)}} (\pounds_{\xi^{(k)}} v)$,  and $\pounds_{\eta^{(k)}} (\pounds_{\eta^{(k)}} v)$ since $\xi^{(k)}$ and $\eta^{(k)}$ are assumed solenoidal), we see that
\begin{align}\nonumber
\langle u_t, v\rangle_{L^2} &= \langle u_0, v\rangle_{L^2}  + \sum_k \int_0^t  \langle u_s,   \pounds_{\xi^{(k)}} v\rangle_{L^2(\Omega)}\  \rmd W_s^{(k)}
 \\ 
&\quad + \int_0^t \left[  \langle u_s,  \pounds_{u_s} v\rangle_{L^2(\Omega)} + \frac{1}{2} \sum_k   \langle u_s,   \pounds_{\xi^{(k)}}( \pounds_{\xi^{(k)}} v)\rangle_{L^2(\Omega)}+ \nu \sum_k   \langle u_s,   \pounds_{\eta^{(k)}}( \pounds_{\eta^{(k)}} v)\rangle_{L^2(\Omega)} \right]  \rmd s.
\end{align}
The resulting equation corresponds exactly with the definition of the weak form \eqref{weakform}, thereby establishing that \eqref{stochweber1} is the solution in the sense of Definition 3 of \cite{CFH17}. 
\vspace{2mm}

\noindent \textbf{Direction 2: Smooth solutions Eq.  \eqref{uEqn1} satisfy the fixed-point problem \eqref{flowNS}, \eqref{stochweber1}.}
Given a smooth solution $u_t$, we may construct a smooth flow $X_t$ solving \eqref{stochweber1}, as well as its back-to-labels map $A_t$ which solves 
 \begin{align} \label{backtolabelsEq}
\rmd_t A_t(x) &+ u_t(x)\cdot \nabla A_t(x) \rmd t +\sum_k \xi^{(k)} \cdot \nabla A_t(x)\circ \rmd W_t^{(k)}+\sqrt{2\nu}\sum_k \eta^{(k)} \cdot \nabla A_t(x)\circ \rmd B_t^{(k)}=0,
 \end{align}
 with data $A_t(x)|_{t=0}= x$.
 This equation is easily established by applying the It\^{o} formula to $A_t\circ X_t={\rm id}$.  
The spatial gradient of the back-to-labels map is then found to solve
  \begin{align}\label{GradbacktolabelsEq}
\rmd_t \nabla A_t(x) &+\pounds_{u_t}^T    \nabla A_t(x) \rmd t +\sum_k  \pounds_{\xi^{(k)}}^T \nabla A_t(x)\circ \rmd W_t^{(k)}+\sqrt{2\nu}\sum_k  \pounds_{\eta^{(k)}}^T  \nabla A_t(x)\circ \rmd B_t^{(k)}=0,
 \end{align}
 with data $\nabla A_t(x)|_{t=0}= \mathbb{I}.$
 Define now $\ol{u}_t:=\ol{u}_t(x)$ from $u_0$, $A_t$ and $\nabla A_t$ by
\be\label{constructbaru}
\ol{u}_t(x) = \mathbb{E}\left[\mathbb{P} (\nabla A_{t}(x))^T u_0(A_{t}(x))\ \Big|\ \mathcal{F}_{t}^{\{W^{(k)}\}}\right], \quad   \mathbf{P} \ a.s.
\ee
We aim to show that $\ol{u}_t$ is a solution to the fixed point problem \eqref{flowNS}, \eqref{stochweber1}.
 To do so, we derive now a stochastic evolution equation for $\ol{u}_t(x)$.  This will require the following two Lemmas
 
 \begin{lemma}\label{thetaLemma}
Let $v\in C(0,T;C^2(\Omega))$ be deterministic.  Then, the process $\theta_t:=v_t\circ A_t$ solves the SPDE
\begin{align}\nonumber
\rmd_t \theta_t   &= \left(\partial_t v|_{A_t}- u_t\cdot  \nabla \theta_t+\frac{1}{2}\sum_k (\xi^{(k)}\cdot \nabla )((\xi^{(k)}\cdot \nabla )\theta_t )+ \nu \sum_k (\eta^{(k)}\cdot \nabla )((\eta^{(k)}\cdot \nabla )\theta_t ) \right)\rmd t \\ \label{thetaEqn}
&\quad -   \sum_k \xi^{(k)} \cdot \nabla \theta_t  \rmd W_t^{(k)} - \sqrt{2\nu}  \sum_k \eta^{(k)} \cdot \nabla \theta_t  \rmd B_t^{(k)}. 
\end{align}
 \end{lemma}
 \begin{proof}
 First, the  It\^{o} form of Eq. \eqref{backtolabelsEq} reads
  \begin{align} \nonumber
\rmd_t A_t(x) &+ u_t(x)\cdot \nabla A_t(x) \rmd t - \frac{1}{2}\sum_k(\xi^{(k)} \cdot \nabla)( (\xi^{(k)} \cdot \nabla) A_t(x)) \rmd t - \nu \sum_k(\eta^{(k)} \cdot \nabla)( (\eta^{(k)} \cdot \nabla) A_t(x)) \rmd t \\
&\qquad +\sum_k \xi^{(k)} \cdot \nabla A_t(x) \rmd W_t^{(k)}+\sqrt{2\nu}\sum_k \eta^{(k)} \cdot \nabla A_t(x) \rmd B_t^{(k)}=0. \label{backtolabelsEqIto}
 \end{align}
 Now, applying the It\^{o} product formula, we have
 \begin{align*}
  \rmd \theta_t &= \partial_t v|_{A_t} \rmd t + \rmd A_t \cdot \nabla v_t|_{A_t} + \frac{1}{2} \rmd \left[ A_t, A_t\right]_t: (\nabla \otimes \nabla v_t)|_{A_t}\\
  &=\partial_t v|_{A_t} \rmd t -  (u_t\cdot \nabla A_t)\cdot \nabla v_t|_{A_t}\rmd t -  \sum_k (\xi^{(k)} \cdot \nabla A_t)\cdot \nabla v_t|_{A_t} \rmd W_t^{(k)}\\
  &\quad - \sqrt{2\nu} \sum_k (\eta^{(k)} \cdot \nabla A_t)\cdot \nabla v_t|_{A_t} \rmd B_t^{(k)}  + \frac{1}{2} \rmd \left[ A_t, A_t\right]_t: (\nabla \otimes \nabla v_t)|_{A_t}\\
 &\quad + \frac{1}{2}\sum_k(\xi^{(k)} \cdot \nabla)( (\xi^{(k)} \cdot \nabla) A_t(x))\cdot \nabla v_t|_{A_t} \rmd t + \nu \sum_k(\eta^{(k)} \cdot \nabla)( (\eta^{(k)} \cdot \nabla) A_t(x))\cdot \nabla v_t|_{A_t} \rmd t.
  \end{align*}
  Using \eqref{backtolabelsEqIto}, we compute the quadratic variation term as 
  \be\label{quadvarA}
\frac{1}{2} \rmd  \left[ A_t, A_t\right]_t=  \frac{1}{2}\sum_k (\xi^{(k)} \cdot \nabla) A_t \otimes  (\xi^{(k)} \cdot \nabla) A_t \rmd t+ \nu \sum_k(\eta^{(k)} \cdot \nabla) A_t \otimes  (\eta^{(k)} \cdot \nabla) A_t \rmd t.
  \ee
  Thus, putting \eqref{quadvarA} together with our calculation of $\rmd \theta_t$, we arrive at the following equation
   \begin{align*}
  \rmd \theta_t   &=\partial_t v_t|_{A_t} \rmd t -  (u_t\cdot \nabla A_t)\cdot \nabla v_t|_{A_t} \rmd t-  \sum_k (\xi^{(k)} \cdot \nabla A_t)\cdot \nabla v_t|_{A_t} \rmd W_t^{(k)}\\
  &\quad - \sqrt{2\nu} \sum_k (\eta^{(k)} \cdot \nabla A_t)\cdot \nabla v_t|_{A_t} \rmd B_t^{(k)} +  \frac{1}{2} \sum_k(\xi^{(k)} \cdot \nabla) A_t \cdot  (\nabla \otimes \nabla v_t)|_{A_t}\cdot  (\xi^{(k)} \cdot \nabla) A_t \rmd t\\
  &\quad  +  \nu \sum_k(\eta^{(k)} \cdot \nabla) A_t \cdot  (\nabla \otimes \nabla v_t)|_{A_t}\cdot  (\eta^{(k)} \cdot \nabla) A_t \rmd t\\
   &\quad + \frac{1}{2}\sum_k(\xi^{(k)} \cdot \nabla)( (\xi^{(k)} \cdot \nabla) A_t(x))\cdot \nabla v_t|_{A_t} \rmd t + \nu \sum_k(\eta^{(k)} \cdot \nabla)( (\eta^{(k)} \cdot \nabla) A_t(x))\cdot \nabla v_t|_{A_t} \rmd t.
  \end{align*}
Using finally the chain rule via both $(u\cdot \nabla A_t)\cdot \nabla v_t= u\cdot \nabla \theta_t$ and the identity
   \begin{align*}
  (\xi \cdot \nabla) A_t \cdot  (\nabla \otimes \nabla v_t)|_{A_t}\cdot  (\xi \cdot \nabla) A_t  + (\xi \cdot \nabla)( (\xi \cdot \nabla) A_t(x)) \cdot \nabla v_t|_{A_t}&=  (\xi\cdot \nabla )((\xi\cdot \nabla )\theta_t ),
  \end{align*}
  we deduce the stated evolution equation \eqref{thetaEqn}.
 \end{proof}

 We now  derive the evolution of the ``Weber velocity" $w_t$, generalizing Thm. 2.2 of \cite{CI08} to multiplicative noise.  It can also be derived as an application of Kunita's formula \eqref{KunflowEqn} above, but we prove it here directly.
 \begin{lemma}\label{wLem}
 Let $v\in C(0,T;C^2(\Omega))$ and  $\theta_t:=v_t\circ A_t$.  The process $w_t:=  (\nabla A_{t}(x))^T\theta_t$ solves the SPDE
 \begin{align}\nonumber
\rmd_t w_t   &+   \left(   \pounds_{u_t}^T w_t- \frac{1}{2}  \sum_k \pounds_{\xi^{(k)}}^T( \pounds_{\xi^{(k)}}^Tw_t)- \nu  \sum_k \pounds_{\eta^{(k)}}^T( \pounds_{\eta^{(k)}}^T w_t)\right)  \rmd t\\
&\qquad\qquad\qquad\qquad\qquad  + \sum_{k}    \pounds_{\xi^{(k)}}^T w_t   \rmd W_t^{(k)} +  \sqrt{2\nu} \sum_{k}    \pounds_{\eta^{(k)}}^T w_t   \rmd B_t^{(k)} =0.\label{wevol}
\end{align}
 \end{lemma}
 \begin{proof}
 First, the   It\^{o} form  of Eqn. \eqref{GradbacktolabelsEq} reads
   \begin{align}\nonumber
\rmd_t \nabla A_t(x) &+\pounds_{u_t}^T    \nabla A_t(x) \rmd t - \frac{1}{2}\sum_k  \pounds_{\xi^{(k)}}^T ( \pounds_{\xi^{(k)}}^T \nabla A_t(x)) \rmd t- \nu\sum_k  \pounds_{\eta^{(k)}}^T ( \pounds_{\eta^{(k)}}^T \nabla A_t(x)) \rmd t\\
&\qquad +\sum_k  \pounds_{\xi^{(k)}}^T \nabla A_t(x) \rmd W_t^{(k)}+\sqrt{2\nu}\sum_k  \pounds_{\eta^{(k)}}^T  \nabla A_t(x) \rmd B_t^{(k)}=0.\label{GradbacktolabelsEqIto}
 \end{align}
Applying the It\^{o} product formula, we have
\begin{align}\label{itoprodW}
\rmd_t w_t   &= \rmd (\nabla A_{t})^T\theta_t+  (\nabla A_{t})^T\rmd \theta_t+ \rmd\left[ (\nabla A_{t})^T, \theta_t \right]_t.
\end{align}
 Using Eqn. \eqref{GradbacktolabelsEqIto} and \eqref{thetaEqn} from Lemma \ref{thetaLemma}, we compute the quadratic variation term to be
 \be
  \rmd\left[ (\nabla A_{t})^T, \theta_t \right]_t= \sum_k  \pounds_{\xi^{(k)}}^T  (\nabla A_t)^T (\xi^{(k)} \cdot \nabla \theta_t  ) \rmd t+{2\nu}\sum_k  \pounds_{\eta^{(k)}}^T  (\nabla A_t)^T (\eta^{(k)} \cdot \nabla \theta_t  ) \rmd t.
 \ee
 We have also that
 \begin{align*}
 \rmd (\nabla A_{t})^T\theta_t&=- \pounds_{u_t}^T    (\nabla A_t)^T\theta_t\rmd t + \frac{1}{2}\sum_k  \pounds_{\xi^{(k)}}^T ( \pounds_{\xi^{(k)}}^T (\nabla A_t)^T) \theta_t\rmd t+ \nu\sum_k  \pounds_{\eta^{(k)}}^T ( \pounds_{\eta^{(k)}}^T \nabla A_t)^T)\theta_t \rmd t\\
&\qquad -\sum_k  \pounds_{\xi^{(k)}}^T (\nabla A_t)^T \theta_t \rmd W_t^{(k)}-\sqrt{2\nu}\sum_k  \pounds_{\eta^{(k)}}^T  (\nabla A_t)^T\theta_t \rmd B_t^{(k)},\\
(\nabla A_{t})^T\rmd \theta_t&= - (\nabla A_{t})^Tu_t\cdot  \nabla \theta_t \rmd t+\frac{1}{2}\sum_k(\nabla A_{t})^T (\xi^{(k)}\cdot \nabla )((\xi^{(k)}\cdot \nabla )\theta_t ) \rmd t\\
&\quad + \nu \sum_k (\nabla A_{t})^T(\eta^{(k)}\cdot \nabla )((\eta^{(k)}\cdot \nabla )\theta_t ) \rmd t -   \sum_k(\nabla A_{t})^T \xi^{(k)} \cdot \nabla \theta_t  \rmd W_t^{(k)} \\ 
&\quad  - \sqrt{2\nu}  \sum_k (\nabla A_{t})^T\eta^{(k)} \cdot \nabla \theta_t  \rmd B_t^{(k)}. 
 \end{align*}
 For any vector field $v$, one has the identity,
\begin{align}\nonumber
   \pounds_{v}^T w_t%&=  (v\cdot \nabla)(\nabla A_{t})^T)\theta_t+  (\nabla A_{t})^T)(v\cdot \nabla)\theta_t + (\nabla v)^T w_t\\
   &=(\pounds_{v}^T (\nabla A_{t})^T)\theta_t +  (\nabla A_{t})^T(v\cdot\nabla) \theta_t\,.
\end{align}
Consequently, the form of the noise and first drift term in \eqref{wevol} are fixed.  Grouping the remaining terms in \eqref{itoprodW} involving $\xi$ and $\eta$,  using the identity  \eqref{lieComps} and then performing some straightforward but tedious computations, we obtain the stated evolution equation \eqref{wevol}.    We remark that  Kunita's formula \eqref{KunflowEqn} can be obtained by pairing (in $L^2(\Omega)$) the equation \eqref{wevol} with an arbitrary solenoidal vector field $v$ and integrating by parts.
 \end{proof}

 Finally, let $\tilde{u}_t(x):=  (\nabla A_{t}(x))^T u_0(A_{t}(x))$ so that $\ol{u}_t = \mathbb{E}[\mathbb{P}(\tilde{u}_t(x))|  \mathcal{F}_{t}^{\{W^{(k)}\}}]$.  Applying Lemma \ref{wLem} to the stochastic Weber velocity $\tilde{u}_t$, projecting onto divergence-free and averaging over the Brownian motions $\{B^{(k)}\}_{k\in \mathbb{N}}$, we deduce that $\ol{u}_t$ solves the following linear SPDE 
\begin{align}\label{linearprob}
 \rmd_t \ol{u}_t + \mathbb{P} \left(   \pounds_{u_t}^T \ol{u}_t - \frac{1}{2}  \sum_k \pounds_{\xi^{(k)}}^T( \pounds_{\xi^{(k)}}^T \ol{u}_t)- \nu  \sum_k \pounds_{\eta^{(k)}}^T( \pounds_{\eta^{(k)}}^T \ol{u}_t)\right)  \rmd t + \sum_{k} \mathbb{P}   \pounds_{\xi^{(k)}}^T \ol{u}_t   \rmd W_t^{(k)} &=0, 
 \end{align}
with initial condition  $\ol{u}_0=u_0$. Clearly, one solution of \eqref{linearprob} is $u_t$ itself.  Uniqueness of the inital value problem for this type of linear stochastic system with regular coefficients follows from the argument given in the proof of Proposition 11  of \cite{CFH17}. Thus, we conclude that $\ol{u}_t\equiv u_t$ for all $t\in [0,T]$ and therefore smooth solutions  $u_t$ of  Eq.  \eqref{uEqn1} solve the fixed-point problem \eqref{flowNS}, \eqref{stochweber1}.
 \end{proof}

\section{Discussion}

In this note, we have considered two classes of stochastic models of Eulerian incompressible fluid flow which differ in their nonlinear transport operators. These two classes may be compared explicitly in their vector field forms, upon defining the stochastic vector field for the transport velocity written in terms of the smooth, invertible, volume-preserving flow map $X_t$ in Eqn. \eqref{flow} as
\be
({ \rmd} X_t) X_t^{-1} := {u_t}dt + {\small \sum}_{k}\xi^{(k)}\circ d W^{(k)}_t\,.
\ee
The stochastic transport operator for the \emph{energy-conserving} stochastic fluid models we have treated here takes the form
\be\label{erg-cons-transport}
\rmd u_t + \mathbb{P}\big( (\rmd X_t X_t^{-1})\cdot\nabla u_t \big) = 0 \,.
\ee
However, stochastic fluid models with this transport operator \emph{do not conserve} the Kelvin circulation, unless the spatial gradients of their correlation eigenvectors $\xi^{(k)}$ all vanish. 

In contrast, the transport operator in the stochastic fluid models we have treated here that \emph{do conserve} Kelvin circulation take the form
\be\label{circ-cons-transport}
\rmd u_t + \mathbb{P}\big( \pounds^T_{(\rmd X_t X_t^{-1})}u_t \big) = 0 
\,,
\ee
where 
\be
\pounds^T_{\rmd X X^{-1}}u := ({\dot{X} X^{-1}}) \cdot \nabla u   
   \,+\, \big(\nabla ({\dot{X} X^{-1}})\big)^T\cdot  u \,. % - (\rmd X_t X_t^{-1})\times {\rm curl} u 
\ee
Thus, the transport operators for the two classes of stochastic Euler equations treated here,  in Eqn. \eqref{erg-cons-transport} which conserve energy and  in Eqn. \eqref{circ-cons-transport} which conserve circulations, only differ by a single term. 

Indeed, we have shown that the stochastic fluid equations in Eqn. \eqref{circ-cons-transport} are \emph{characterized} by the property that circulations are conserved (pathwise in case of Euler-type models and in mean for Navier-Stokes-type) on smooth solutions.  Brownian forces enter into these equations as a novel type of multiplicative noise; which involves the Lie derivative of the circulation velocity along the spatial correlation eigenvectors of the noise.  This structure has geometric significance which ensures that the stochastic equations retain the Lagrangian properties of circulation, vorticity and helicity which their deterministic counterparts possess.
However, stochastic fluid models with the transport operator in Eqn. \eqref{circ-cons-transport} turn out not to conserve \emph{energy}, unless unless the spatial gradients of their  correlation eigenvectors $\xi^{(k)}$ in the cylindrical noise all vanish.  

The difference between these two classes of stochastic models may appear small, especially since their transport operators exactly coincide in the deterministic case, where they each conserve both energy and Kelvin circulation. However, we have found that this difference has profound effects in the conservation properties of the  stochastic fluid models treated here.
Thus, the introduction of gradients into the spatial correlations of the cylindrical noise in these two classes of stochastic fluid models has introduced a sort of ``Sophie's choice" between conservation of either energy, or circulation, but not both.

 This comparison is summarized explicitly in the following table.
\vspace{2mm}

\begin{center}
\begin{tabular}{ |p{3cm}|p{3cm}|p{3cm}|p{3cm}|  }
 \hline
 \multicolumn{4}{|c|}{ Comparison of Energy and Circulation Properties of Stochastic Fluid Models} \\
 \hline
Euler\phantom{adadasfffdd}  Circulation Thm.& 
Euler\phantom{adadasfffdd} Energy Thm. &
Navier-Stokes \phantom{ada} Circulation Thm.&
Navier-Stokes\phantom{ada}  Energy Thm.\\
 \hline
Eqn. \eqref{uEqn1} with\phantom{ada}  \eqref{feq} and \eqref{sigeq}   & 
Eqn. \eqref{uEqn1} with\phantom{ada}  \eqref{Econseqnsf} and \eqref{EconseqnsSig}    &
Eqn. \eqref{uEqn1} with\phantom{ada}  \eqref{feqNS} and \eqref{sigeqNS}  &   
Eqn. \eqref{uEqn1} with\phantom{ada}  \eqref{EconseqnsNSf} and \eqref{EconseqnsNSSig}   \\
 \hline
\end{tabular}
\end{center}
\vspace{4mm}

With these models in hand, we must address the following important questions: what physical systems do they represent; what insights do they yield; and how can they be exploited in practice?
To address the first question, we mention the recent work of \cite{CGH17} which shows that the stochastic Euler-Poincar\'{e} equations (Eqn. \eqref{uEqn1} with  \eqref{feq} \& \eqref{sigeq}) arise naturally upon representing the deterministic Lagrangian flow map as a composition of smooth maps with two different time scales. The first map has slowly varying time dependence. It is followed by composition with the second map which has rapidly fluctuating time dependence, with zero mean when homogenized over the rapid time scale.  When dissipation is important, the corresponding  Navier-Stokes-Poincar\'{e} equations  (Eqn. \eqref{uEqn1} with   \eqref{feqNS} \& \eqref{sigeqNS}) arise from similar considerations. The result of \cite{CGH17} shows that this stochastic model is similar in spirit to a deterministic regularization of the Navier-Stokes equation called the LANS-$\alpha$ model, which has been proposed as a model for large-scale turbulence and also preserves a certain Kelvin circulation theorem \cite{HMR98,CFHOTW98,FHT01,FHT02,HJKLTW05}.

The above considerations motivate the utility of circulation-theorem preserving stochastic models as
reduced descriptions of nonlinear dynamical systems which account for effects  of the small, rapid, unresolvable scales of fluid motion on the variability of computationally resolvable.  See, respectively, \cite{CCHOS18a} and \cite{CCHOS18} for computational investigations of the Navier-Stokes-Poincar\'{e} equations in two dimensions for regions with fixed boundaries and for a 2-layer quasi-geostrophic model. See  also \cite{GH18} for a recent review, and see \cite{GH18a} for discussions of stochastic fluid models with non-stationary statistics.

On the other hand, for certain applications (depending on what observable the stochastic solution is intended to describe) it may be more important to enforce a pathwise energy balance.  In this case, the models treated here which preserve the corresponding energy theorems (Eqn. \eqref{uEqn1} with \eqref{Econseqnsf} \& \eqref{EconseqnsSig}  or  \eqref{EconseqnsNSf} \& \eqref{EconseqnsNSSig}) may be very useful.
In particular, another deterministic regularization of the Navier-Stokes equation due to Leray \cite{Leray1934} called the Leray-$\alpha$ model, which conserves energy in the absence of viscosity, has also been developed for simulations of turbulence and studied numerically in comparison with the LANS-$\alpha$ model \cite{GH2002,GH2003}. 

It is only when the noise-coefficients are spatially homogeneous that these two models simultaneously preserve their respective energy and circulation theorems. Thus, as is typical in the modeling business, an application-dependent choice must be made whenever implementing these SPDEs as a practical reduced description.  These issues are currently being explored and remain the subject of active and ongoing research.

\appendix

\section{Geometric background and notation} \label{app-GeometBckgrnd}

We discuss Kelvin's circulation theorem from a geometric viewpoint. To begin, let $X_t\in {\rm SDiff}(\Omega)$ be a volume-preserving diffeomorphism (i.e., smooth invertible flow, whose inverse is also smooth) which maps  the manifold without boundaries, $\Omega\subset \mathbb{R}^d$ onto itself. Introduce the \emph{transport velocity}, $u_t$, as a vector field, 
\[
u_t:= u(x,t):\Omega\times [0,T]\to \mathfrak{X}(\mathbb{R}^d)
\,,\]
where $\mathfrak{X}(\mathbb{R}^d)$ is the space of volume-preserving vector fields defined over $\mathbb{R}^d$; so that $\nabla \cdot u_t = 0$. 
Next, define the corresponding \emph{circulation velocity}, $u_t^\flat$, which appears in the integrand of Kelvin's theorem,  
\[
u_t^\flat:= u^\flat(x,t):\Omega\times [0,T]\to \Lambda^1(\mathbb{R}^d)\,.
\]
Here $u_t^\flat$ is in the space of  1-forms $ \Lambda^1(\mathbb{R}^d)$ dual to the divergence-free transport velocity vector fields, $u_t\in \mathfrak{X}(\mathbb{R}^d)$, under the $L^2$ pairing between the Lie algebra of vector fields and its dual, 
\[
\langle \,\cdot\,,\, \cdot\,\rangle:
\Lambda^1(\mathbb{R}^d) \times \mathfrak{X}(\mathbb{R}^d)\to \mathbb{R}\,,
\]
on the domain of flow $\Omega$. Here, the familiar musical operations flat ($\flat$) and its inverse sharp ($\sharp$) essentially lower and raise vector indices, respectively, although no Riemannian metric will be needed. For example, the operation $\flat: \mathfrak{X}(\mathbb{R}^d) \to  \Lambda^1(\mathbb{R}^d)$ maps a vector field into a 1-form, and vice versa for $\sharp$ (so that $(u_t^\flat)^\sharp=u_t$, for example). The musical notation which distinguishes between $u_t^\flat$ and $u_t$ helps one make proper mathematical sense of the operations of divergence, Lie derivative, Leray-Hodge projection, etc.

Suppose the transport and circulation velocities $u_t$ and $u_t^\flat$, respectively, together solve the incompressible Euler equations,
written in vector form as
\begin{align}
\begin{split}\label{Euler1form}
\partial_t u_t^\flat + (u_t\cdot  \nabla)  u_t^\flat &= - {\nabla}  p_t 
\end{split}
\end{align}
with scalar pressure function $p_t$, determined by solving the Poisson equation
$-\Delta p_t = (\nabla\otimes \nabla): (u_t\otimes    u_t^\flat )$.
The \emph{Kelvin theorem} in Eqn. \eqref{KelvinTheorem} now states that any smooth Euler solution $u_t$ has the property that for all loops $\Gamma \subset \Omega$, the circulation integral satisfies, 
\be\nonumber
\oint_{X_t(\Gamma)} \!\!\! u_t^\flat \,= \oint_{X_0(\Gamma)} \!\!\! u_0^\flat
\ee
where the time-dependent Lagrangian flow map $X_t$ with $X_{0} ={\rm id}$ is obtained by integrating the vector field 
\[
\dot{X}_t= u_t ( X_t )=: X_t^* u_t
\,,\]
%%%%%%%%%%%%%%%%%%%%%
where the asterix on $X_t^*$ denotes pull back by the smooth invertible map $X_t$.  Consequently, the transport velocity vector field in the Eulerian representation is given by 
\[
u_t ={\dot{X}_t X_t^{-1}}\in \mathfrak{X}(\mathbb{R}^d)
\,,\]
in which the right action on the tangent vector $\dot{X}_t$ by the inverse map $X_t^{-1}$ (shown as concatenation from the right) translates the tangent vector along the Lagrangian path back to the identity. Thus, the Eulerian  transport velocity vector field $u_t \in \mathfrak{X}(\mathbb{R}^d)$ is right-invariant. That is, $u_t ={\dot{X}_t X_t^{-1}}$ is invariant under the action of the diffeomorphisms from the right, upon transforming $X_t\to X_t \bar{X}_t$ for any other volume-preserving diffeomorphism, $\bar{X}_t\in {\rm SDiff} (\mathbb{R}^d)$. As we shall see, right-invariance is the key to understanding the Kelvin circulation theorem from the viewpoint of Noether's theorem.

The  Kelvin Theorem in \eqref{KelvinTheorem} offers some insight into the geometric meaning of the Euler fluid equations. 
 In the geometric notation introduced above, the calculation in Eq. \eqref{transthm} may be validated as
   \begin{align}\label{transthm-geom}
   \begin{split}
   \frac{\rmd}{\rmd t}\oint_{X_t(\Gamma)} u_t^\flat
   &:=
   \oint_{X_0(\Gamma)} \frac{\rmd}{\rmd t} \big(X_t^* u_t^\flat \big)
   \\&=
   \oint_{X_0(\Gamma)} X_t^* \Big(\big(\partial_t  + \pounds_{\dot{X}_t X_t^{-1}}\big)u_t^\flat\Big)
   \\&=
   \oint_{X_t(\Gamma)} \big(\partial_t  + \pounds_{\dot{X}_t X_t^{-1}}\big)u_t^\flat
   \\&=
   \oint_{X_t(\Gamma)} \big(\partial_t u_t + \big(({\dot{X}_t X_t^{-1}}) \cdot \nabla \big)u_t   
   \,+\, \nabla ({\dot{X}_t X_t^{-1}})^T\cdot  u_t  \big) \cdot \rmd \ell = 0\,.
   \end{split}
   \end{align}
  In the second line of this calculation, we have used the formula \cite{MarsdenHughes94}
  \be\label{DynLieDerivFormula}
   \frac{\rmd}{\rmd t} \big(X_t^* (u_t^\flat)\big) 
   =
   X_t^*\Big(\big(\partial_t  + \pounds_{\dot{X}_t X_t^{-1}}\big)u_t^\flat \Big),
  \ee
  in which the pull-back by the flow map $X_t$ acting on the Lie derivative $\pounds_{\dot{X}_t X_t^{-1}}u_t^\flat$ of a 1-form $u_t^\flat$ with respect to the vector field ${\dot{X}_t X_t^{-1}}$ is defined as the time derivative of the pull-back of the 1-form $u_t^\flat$ by the flow map $X_t$. 
  In the third line above in \eqref{transthm-geom}, transforming the integrand back into fixed Eulerian coordinates yields the Lie derivative itself, defined as the tangent of the pull-back, evaluated at the identity; which, as a vector expression is given by, 
  \be\label{LieDerivFormula}
   \pounds_{\dot{X}_t X_t^{-1}}u_t^\flat:= \Big[ \frac{\rmd}{\rmd t} \big(X_t^* u_t^\flat \big) \Big]_{\rm id}
   = \big(({\dot{X}_t X_t^{-1}}) \cdot \nabla \big)u_t   
   \,+\, \nabla ({\dot{X}_t X_t^{-1}})^T\cdot  u_t  \big) \cdot \rmd \ell \,,
  \ee
thereby finishing the calculation. 
%%%%%%%%%%%%%%%%%%%%%

  Now, in comparing Eq. \eqref{transthm}  with Eq. \eqref{LieDerivFormula}, we realize that the geometric meaning of the Euler fluid equations was disguised in Eq. \eqref{transthm}, by not distinguishing between the transport velocity vector field and the circulation velocity 1-form. Of course, this distinction is unnecessary in Euclidean coordinates. However, even in Euclidean coordinates we will benefit in what follows by keeping track of this distinction. In particular, the properties of the Lie derivative will be very useful to us in what follows; and the Lie derivative of a 1-form is not the same as  the Lie derivative of a vector field.

The Lie derivative of one (right-invariant, Eulerian) vector field $w$ by another one $\xi$ is defined by the following well known formula, see, e.g., \cite{HoMaRa1998, HoScSt2009, ArnoldKhesin1998},
\be
- \pounds_{\xi} w= -  \big( (\xi\cdot \nabla) {\boldsymbol w} - (w\cdot \nabla) {\boldsymbol\xi} \big) \cdot {\boldsymbol \nabla} :=  [\xi, w] =:  {\rm ad}_\xi w
\,.
\label{Lie-brkt}
\ee
In contrast, the Lie derivative of a 1-form $v^\flat$ by the vector field $\xi$ is given as in the calculation \eqref{transthm-geom} above as 
\be
\pounds_{\xi} v^\flat := \big( (\xi\cdot \nabla) {\boldsymbol v} +  ({\boldsymbol \nabla} \xi)^T \cdot v \big) \cdot \rmd {\boldsymbol \ell} =: {\rm ad}^*_\xi v^\flat
\,.
\label{Lie-deriv-1form}
\ee
In the pairing $\langle \cdot,\cdot \rangle_{L^2(\Omega)}$ with respect to the standard $L^2(\Omega)$ inner product, the operations ${\rm ad}$ and ${\rm ad}^*$ are dual to each other, being related by \cite{HoMaRa1998, HoScSt2009}
\be\label{ad-adstar}
\langle {\rm ad}^*_{\xi} v^\flat, w\rangle_{L^2(\Omega)}= \langle v^\flat, {\rm ad}_\xi w\rangle_{L^2(\Omega)}.
\ee
To simplify notation in what follows, we now define the adjoint operator $\pounds_{\xi}^T$ by the identity,
\be\label{adjointform}
( \pounds_{\xi}^T v, w )_{L^2(\Omega)} := \langle (\pounds_{\xi}^T v)^\flat, w\rangle_{L^2(\Omega)}
= \langle \pounds_{\xi} v^\flat, w\rangle_{L^2(\Omega)} = - \langle v, \pounds_{\xi} w\rangle_{L^2(\Omega)},
\ee
where the round brackets  $( \cdot,\cdot )_{L^2(\Omega)}$  denote the usual $L^2$ integral of the dot product of vector-valued functions. 
Consequently, $(\pounds_{\xi}^T v)^\flat = \pounds_{\xi}v^\flat$, upon identifying corresponding terms. This relation follows due to the nondegeneracy of the $L^2(\Omega)$ pairing for a manifold without boundaries. It may also be verified by substituting \eqref{Lie-brkt} into \eqref{adjointform} and integrating by parts.

Upon taking the $\sharp$ of Eqn. \eqref{Euler1form} to transform it from 1-forms to vector fields and applying the Leray-Hodge projection  $\mathbb{P} $, it becomes
$
 \rmd u_t + \mathbb{P}  ( \pounds_{u_t}^T u_t) \rmd t  =0.
$
Thus, the corresponding equation for the vector field $u_t = \boldsymbol{u}_t \cdot  \boldsymbol{\nabla}$ can be expressed as
\be \nonumber
 \rmd u_t +  \mathbb{P} \Big( {\rm ad}^\dagger_{u_t} u_t\Big)  \,  \rmd t = 0\,,
\ee
where the binary operation among vector fields ${\rm ad}^\dagger: \mathfrak{X}\times \mathfrak{X}\to \mathfrak{X}$ is defined for vector fields $\xi$ and $v$ by
\be\label{ad-dagger}
{\rm ad}^\dagger_\xi v := ({\rm ad}^*_{\xi} v^\flat) = (\pounds_{\xi} v^\flat) =: \pounds_{\xi}^T v 
\,.\ee
Having identified 
\[
{\rm ad}^*_{\xi} v^\flat 
= \pounds_{\xi}v^\flat =: (\pounds_{\xi}^T v)^\flat
\quad \hbox{and}\quad
{\rm ad}^\dagger_\xi v = \pounds_{\xi}^T v\,, 
\]
from equations \eqref{ad-adstar}, \eqref{adjointform} and \eqref{ad-dagger}, we see that the musical notations sharp ($\sharp$) and flat ($\flat$) can now be replaced by the simpler $\pounds_{\xi}^T$ notation. Namely,  in what follows, we will distinguish notationally between components of Lie-derivative operations on vector fields and 1-forms as,
 \be\label{lietrans}
- \pounds_{\xi} w := [\xi,w] = (\xi\cdot \nabla) w - (w\cdot \nabla) \xi
\quad\hbox{and}\quad
(\pounds_{\xi}v^\flat)^\sharp  :=  \pounds_\xi^T v := \xi \cdot \nabla v+ \nabla \xi \cdot v
\,.
 \ee
This notation distinguishes between divergence free vector fields and their $L^2$-dual 1-forms only by whether the action of vector fields $\xi$ on them appears as $\pounds_\xi$ or $\pounds_\xi^T$. We note that the operation $\pounds_{\xi}^T$ is denoted as $\mathcal{B}(\xi,\,\cdot\,)$ in \cite{ArnoldKhesin1998}, as may be identified in the following relation,
\[
(\pounds_{\xi}^T v, w)_{L^2(\Omega)} = (\mathcal{B}(\xi,v), w)_{L^2(\Omega)}
\,.
\]
The operator $\mathcal{B}$ in \cite{ArnoldKhesin1998} is distinct from $B(w,v):= \mathbb{P}(w \cdot \nabla v)$ (see e.g. \cite{CF88}) introduced for Eqn. \eqref{Econseqnsf}.

\begin{rem}[Commutator in three-dimensional Euclidean space]
As we see above in Eqn. \eqref{lietrans}, the commutator of two (right-invariant) vector fields is (minus) their Lie derivative.
The commutator of divergence-free vector fields in a three-dimensional Euclidean
space $\mathbb{R}^3$ is given by the formula 
\be\nonumber
- \pounds_{\xi} w:=[\xi,w] = {\rm curl}(\xi \times w)\,,
\ee
where $\xi \times w$ is the cross product.  Hence, we may rewite the relations in Eqn. \eqref{adjointform} in this notation as
   \begin{align}\nonumber
   \begin{split}
( \pounds_{\xi}^T v, w )_{L^2(\Omega)} &=  \big( v, [\xi,w] \big)_{L^2(\Omega)}
 =  \big( v, {\rm curl}(\xi \times w) \big)_{L^2(\Omega)}
 \\& =  
      \big( {\rm curl} v, \xi \times w \big)_{L^2(\Omega)}
  =  \big(  - \xi \times {\rm curl}v , w \big)_{L^2(\Omega)}.
   \end{split}
\end{align}
Thus, we find, in ordinary vector notation, 
\be \nonumber
\pounds_{\xi}^T v = - \,\xi \times {\rm curl}v  
\,,
\ee
modulo a gradient term, since $ \nabla \cdot w = 0$. 

\end{rem}

%%%%%%%%%%%%%%%%%%%%%%%%%%%%%%%

\section{Variational Principle for the Stochastic Euler-Poincar\'{e} equations} \label{appendVar}

In this appendix, we treat only the formal aspects of stochastic variational principles in infinite dimensions, for the purpose of modelling time-dependent spatial correlations. As discussed in Remark \ref{localExRem}, some of the fundamental questions in analysis for the stochastic 3D Euler--Poincar\'{e} fluid model have been answered in \cite{CFH17}, who proved local in time existence, uniqueness and well posedness of their solutions in regular spaces, as well as a Beale-Kato-Majda blow-up criterion for these equations. These are precisely the same analytical properties as for the deterministic 3D Euler fluid equations. Thus, in this case, introducing stochasticity that preserved the geometric properties of the Euler fluid equations also preserved their analytical properties. The corresponding questions still remain open for the other stochastic fluid models discussed here.

\subsection{The stochastic Hamilton--Pontryagin variational principle \cite{GH18a}}

We proceed formally here and below to derive the \emph{stochastic Euler--Poincar\'{e} equations} in \eqref{e1}, by considering the \textit{reduced stochastic Hamilton-Pontryagin  (RSHP) principle} in which the Lagrangian path in Eqn. \eqref{flow} in Theorem \ref{ThmKE} is written in Eulerian coordinates and imposed as a constraint on variations in the Eulerian representation of Hamilton's principle, as
\begin{equation}\label{Reduced_SVP}
\delta \int_0^T \Big[\, l (u_t)dt
+ \Big\langle m, ({ \rmd}X_t)X_t^{-1} - u_tdt 
- \sum_{k}\xi^{(k)}\circ d W^{(k)}_t \Big\rangle_{L^2(\Omega)}  \,\Big]=0\,,
\end{equation}
with respect to variations $ \delta u_t, \delta X_t, \delta m $, for the Lagrangian functional  $l (u_t)$.
This is the reduced stochastic Hamilton-Pontryagin (RSHP) principle found in \cite{GH18a}. Its Eulerian stationarity conditions are
\begin{align*}
\delta m:&\quad ({ \rmd} X_t) X_t^{-1} = {u_t}dt + {\small \sum}_{k}\xi^{(k)}\circ d W^{(k)}_t
\,,\\
\delta u_t:& \quad  
\frac{\delta l }{\delta {u_t}}=m
\,,\\
\delta X_t:& \quad  
{ \rmd} m+ \operatorname{ad}^*_{({ \rmd} X_t) X_t^{-1}}m=0,
\end{align*}
where we have applied the formula for integration by parts for a Stratonovich stochastic process \cite{P2005} in computing the dynamics of the Lagrange multiplier, $m$. In this computation, we have also used the relation,
\[
\delta\big( ( \rmd X_t) X_t^{-1} \big) = \rmd w - {\rm ad}_{( \rmd X_t) X_t^{-1}}w
\,,\quad\hbox{for the vector field}\quad
w = ( \delta X_t) X_t^{-1}
\]
and dropped the endpoint term $\langle m , w \rangle|_0^T$, since the variation $\delta X $ vanishes at the endpoints of interval $[0,T]$. \smallskip

In the Euler fluid case, the Lagrangian is the fluid kinetic energy
 \[
l (u_t)=\frac12 \|u_t\|_{L^2(\Omega)}^2
\]
 and its variation with respect to the velocity vector field is given by the circulation 1-form,
\[
m = \frac{\delta l }{\delta {u_t}} = u_t^\flat
\,.
\]
Now, taking the $\sharp$ of the variational equation for $m$ above and using the divergence free property of the vector field $({ \rmd} X_t)X_t^{-1} $ in the pairing, yields the velocity vector-field equation,
\[
0 =   {\rmd} (\mathbb{P} m^\sharp) + \mathbb{P} ( \pounds_{ ( \rmd X_t) X_t^{-1} } m)^\sharp 
   =  {\rmd} u_t   +  \mathbb{P} ( \pounds^T_{ ( \rmd X_t) X_t^{-1} } u_t ) \,.
\]
Thus, for the Euler case, the stochastic RSHP principle in \eqref{Reduced_SVP} yields the stochastic Euler--Poincar\'{e} motion equation in \eqref{e1}. 

The Eulerian vector field $({ \rmd} X_t) \,X_t^{-1} \in \mathfrak{X}(\mathbb{R}^d)$ is invariant under the action of the diffeomorphisms from the right, given by $X_t\to X_t \bar{X}$ for any fixed volume-preserving diffeomorphism $\bar{X}\in {\rm SDiff} (\mathbb{R}^d)$. Since the motion of a Lagrangian trajectory is given by applying $X_t$ to an initial condition $x_0$, this symmetry simply corresponds to well-known invariance of the Eulerian fluid velocity vector field $u_t$ under relabelling of the Lagrangian coordinates as $x_0\to \bar{X}x_0$. 

\subsection{\bf Noether's theorem and preservation of Kelvin circulation.}
The endpoint term arising from integration by parts in the RSHP variational principle is $\langle m , w \rangle$, as shown above. Vanishing of the endpoint term leads to the variational equations of motion. However, according to Noether's theorem, if $\delta S = 0$ due to invariance of the Lagrangian under a Lie symmetry transformation, then the endpoint term will keep its value under the evolution governed by variational equations. In the present case, the right-invariant vector field $w$ generates an arbitrary time-independent diffeomorphism of the reference flow domain, under which the Lagrangian is invariant, since the Eulerian representation is invariant under a volume-preserving diffeomorphism of the Lagrangian parcel labels.  

In the Euler fluid case, $m$ is a 1-form density and the quantity $m/D$ is a 1-form, although we can ignore the difference, since $D=1$ results as the Jacobian for the flow map generated by a divergence-free vector field. Thus, we can regard $m = u_t^\flat = \mathbf{u} \cdot d\mathbf{x}$ as simply a 1-form, which is evolving by coadjoint action on it by the diffeomorphism $X_t$, so that it satisfies
\[
 X_t^* \big( { \rmd} m+ \operatorname{ad}^*_{({ \rmd} X_t) X_t^{-1}}m \big) = \frac{d}{dt}\big( X_t^*  m \big)
 = - X_t^* (dp)
\,,
\]
where $X_t^* $ is the pullback of the Lagrangian flow and we have introduced $-dp=-{\boldsymbol \nabla} p\cdot d {\bold x}$ to account for incompressibility of $m^\sharp$. 
The previous equation implies that the integral of the 1-form $m$ around any loop that moves with the flow is  constant, as a result of its RSHP equation of motion. Thus, by Noether's theorem, invariance of the Eulerian form of the fluid Lagrangian under fluid particle relabelling implies preservation of Kelvin's circulation integral.

\begin{rem}[Conservation of helicity] \label{helicityRem}
The previous equation is equivalent to the Eulerian expression,
\[
\rmd m = - \pounds_{({ \rmd} X_t) X_t^{-1}}m - dp\,.
\]
Consequently, the stochastic evolution of the helicity density (a 3-form) is given by
\begin{align*}
\rmd (m \wedge dm) &= - \Big(\pounds_{({ \rmd} X_t) X_t^{-1}}m + dp \Big)\wedge dm)
- m \wedge \Big(\pounds_{({ \rmd} X_t) X_t^{-1}}dm\Big)
\\&= - \pounds_{({ \rmd} X_t) X_t^{-1}}(m \wedge dm) - d\big(p\, dm \big)
\\ &= 
- \, {\rm div} \Big( \big( ({ \rmd} X_t) X_t^{-1}\big) (\mathbf{u}\cdot {\rm curl }\mathbf{u})
+  p\, {\rm curl }\mathbf{u} \Big) d^3x\,.
\end{align*}
For homogeneous boundary conditions, this implies the conservation of the \emph{helicity integral},
\[
\rmd \int_{\Omega} (m \wedge dm)
=
\rmd \int_{\Omega} \mathbf{u}\cdot {\rm curl }\mathbf{u} \,d^3x
= 0
\,,
\]
which is interpreted as the conservation of the average self-linking number of vorticity field lines, \cite{AK99,S18}.
\end{rem}

\subsection{\bf Passing to the Lie-Poisson Hamiltonian formulation.}
The Noether quantity also plays an important geometric role on the Hamiltonian side. The reduced Legendre transformation in the Eulerian representation is given by, cf. Eqn. \eqref{Reduced_SVP},
\begin{align}\label{LegXform}
\begin{split}
h(m) &= \Big\langle m , ({\rmd}X_t)X_t^{-1}  \Big\rangle_{L^2}  
- \Big[\, l (u_t)dt
+ \Big\langle m, ({\rmd}X_t)X_t^{-1} - u_tdt 
- \sum_{k}\xi^{(k)}\circ d W^{(k)}_t \Big\rangle_{L^2}  \,\Big] 
\\&=
\Big\langle m, u_tdt + \sum_{k}\xi^{(k)}\circ d W^{(k)}_t \Big\rangle_{L^2}   - l (u_t)dt
\\&=
\frac12 \Big\langle m, m^\sharp \Big\rangle_{L^2}  dt + \Big\langle m, \sum_{k}\xi^{(k)}\Big\rangle_{L^2}  \circ d W^{(k)}_t 
\,,
\end{split}
\end{align}
where we have used $l (u_t) = \frac12 \langle u^\flat , u \rangle_{L^2}  = \frac12 \langle m, m^\sharp \rangle_{L^2} $ and the symmetry of the pairing $\langle \,\cdot\,,\,\cdot\,\rangle_{L^2} $ to simplify and regroup terms in the final step of deriving the reduced Hamiltonian, $h(m)$. We note that the stochastic part of the Hamiltonian $h(m)$ in \eqref{LegXform} couples the noise to the momentum map by $L^2$ pairing. The variational derivative of $h(m)$ with respect to $m$ returns the original stochastic Eulerian vector field, 
\[
\frac{\delta h}{\delta m} = m^\sharp dt + \sum_{k}\xi^{(k)} \circ d W^{(k)}_t = ({\rmd}X_t)X_t^{-1}\,.
\]
Finally, we may rearrange the Euler fluid motion equation into the Lie-Poisson Hamiltonian form \cite{HoMaRa1998}
\[
\rmd m = -  \operatorname{ad}^*_{\delta h/\delta m} m = \big\{m , h(m) \big\}\,,
\]
in which the stochastic Hamiltonian is given above in the last line in Eqn. \eqref{LegXform}, and the Lie-Poisson bracket for functionals $f$ and $h$ is defined by
\be\label{Ham-form}
\rmd  f(m) =
\big\{f(m) , h(m) \big\} := - \left\langle m \,,\, {\rm ad}_{\delta h/\delta m} \frac{\delta f}{\delta m} \right\rangle_{L^2} 
= - \left\langle m \,,\, \left[\frac{\delta h}{\delta m}\,,\, \frac{\delta f}{\delta m} \right]\right\rangle_{L^2} 
.
\ee
This Lie-Poisson bracket satisfies the Jacobi identity, because it is a linear functional of the Lie bracket for the Lie algebra of divergence-free vector fields, which is known to satisfy the Jacobi identity.

Nonconservation of the deterministic energy under this Hamiltonian dynamics can be checked easily by setting $f(m)=\frac12 \langle m, m^\sharp \rangle = \frac12\|u_t\|^2_{L^2}$ in the Eqn. \eqref{Ham-form} and denoting $\Xi:=\sum_{k}\xi^{(k)} \circ d W^{(k)}_t$, to find, 
\begin{align*}
\rmd \frac12 \langle m, m^\sharp \rangle_{L^2}  &= \big\{\frac12 \langle m, m^\sharp \rangle_{L^2(\Omega)} ,h(m)\big\}
= - \big\langle     {\rm ad}^*_{(\rmd X_t X_t^{-1})} m, u   \big\rangle_{L^2} 
\\&=
- \big\langle \pounds^T_{\Xi} u\,,\, u \big\rangle_{L^2} 
= \big\langle \Xi \times {\rm curl} u\,,\, u \big\rangle_{L^2}
= \big\langle - u \times {\rm curl} u\,,\, \Xi \big\rangle_{L^2}
\\&=
 \big\langle u \cdot \nabla u\,,\, \Xi \big\rangle_{L^2}
 = - \int_\Omega u_i \Xi^i_{,\,j} u^j\,d^3x\,,
\end{align*}
where we have used the divergence-free property twice in the last line, when integrating by parts. 
This result is the Stratonovich version of Eqn. \eqref{erg-noncons} when viscosity $\nu$ is absent. Namely, the original deterministic fluid kinetic energy is not conserved under the evolution of the circulation conserving stochastic fluid model, unless the spatial gradients of the correlation eigenvectors $\xi^{(k)}$ vanish.

\begin{rem}[Purely stochastic passive 1-form transport]
If we simply drop the fluid kinetic energy in the total Hamiltonian $h(m)$ in Eqn. \eqref{LegXform}, then only the stochastic part would remain. Consequently, the Lie-Poisson bracket in \eqref{Ham-form} would produce a \emph{linear} passive 1-form transport equation  given by
\[
{ \rmd} m+ \operatorname{ad}^*_{\big(\sum_{k}\xi^{(k)} \circ\, d W^{(k)}_t \big)}m= - dp,
\]
where $p$ is determined by requiring that the gauge $\nabla \cdot A_t=0$ be preserved.  In our other notation, the above equation can be written for a 1-form $A_t$  as
\be\label{Aeqn}
\rmd A_t - \sum_k \mathbb{P}( \pounds_{\xi^{(k)}}^T A_t) \circ \rmd W_t^{(k)} = 0 \,,
\ee
where $\pounds_{\xi^{(k)}}^T A_t = \xi^{(k)}\cdot \nabla A_t + (\nabla \xi^{(k)})^T \cdot A_t$,  in vector notation.
Eqn. \eqref{Aeqn} is the dual problem to the passive Lie-transport equation for the vector field $B_t= {\rm curl} A_t$, 
\be\label{Beqn}
\rmd B_t +  \sum_k \pounds_{\xi^{(k)}} B_t \circ \rmd W_t^{(k)}= 0\,,
\ee
%which may be written in vector notation as
%\[
%\rmd B_t - \sum_k {\rm curl} \big( {\xi^{(k)}}\times B_t \big)\circ \rmd W_t^{(k)} = 0\,.
%\]
where $ \pounds_{\xi^{(k)}} B_t = (\xi^{(k)}\cdot \nabla) B_t - (B_t\cdot \nabla) \xi^{(k)}$ in vector notation.  Note that, since $\nabla \cdot A_t =0$, the field $A_t$ can be recovered uniquely from $B_t$ via the Biot-Savart law $A_t = (-\Delta )^{-1} \curl(B_t)$.  In parallel with Remark \ref{helicityRem}, for this linear stochastic transport problem, the \emph{magnetic helicity} $ \langle A_t, B_t\rangle_{L^2(\Omega)}$ is conserved pathwise.  

The equation \eqref{Beqn} is known as the Kazantsev-Kraichnan model of kinematic dynamo, in which $B_t$ represents a transported magnetic field by a white-in-time Gaussian advecting velocity which is typically assumed to be spatially rough  \cite{Kaz68,Kr67,Kr68}. Not unexpectedly, when the noise correlates  $\{\xi^{(k)}\}_{k\in \mathbb{N}}$ are smooth, the Kelvin theorem for Eq. \eqref{Aeqn} preserves the circulation around closed loops which are transported along stochastic Lagrangian paths in the Stratonovich sense. In this setting, the circulation integral represents the gauge-invariant magnetic flux and the conservation law corresponds to Alfv\'{e}n's theorem. The stochastically propagating closed loops must each retain its linkage number; since diffeomorphisms cannot change the topology of a curve embedded in the flow, even if the flow has a stochastic time dependence.   This may fail  to be true in the Kazantsev-Kraichnan model  in which the fields  $\{\xi^{(k)}\}_{k\in \mathbb{N}}$ are assumed to be only H\"{o}lder continuous $C^\alpha(\Omega)$ with exponent $\alpha\in (0,1)$.  In this case, Lagrangian trajectories in fixed realizations of the advecting Gaussian velocity may become non-unique and the phenomenon of spontaneous stochasticity \cite{BGJ98,D17} must be accounted for when discussing Lagrangian transport properties, see  \cite{DE17,Ey2007}.
\end{rem}

\bigskip

\noindent {\bf Acknowledgments}  
We are enormously grateful for many helpful and  inspiring discussions with P. Constantin, A.B. Cruzeiro, D. Crisan, G.L. Eyink, F. Flandoli, J.-M. Leahy, E. M\'emin, H.Q. Nguyen, T.S. Ratiu, V. Resseguier and S. Takao.
Research of TD is partially supported by NSF-DMS grant 1703997. DDH is grateful for partial support by the EPSRC Standard Grant EP/N023781/1.

\end{document}